\documentclass{amsart} 

\usepackage{amsthm, amssymb, amsmath, amscd}
\usepackage{eurosym}
\usepackage{amsfonts}
\usepackage{amsmath}
\usepackage{setspace}
\usepackage{graphicx}

\usepackage{color}

\usepackage[all,cmtip]{xy}

\usepackage{indentfirst}

\newtheorem{theorem}{Theorem}[section]
\newtheorem{theorem*}{Theorem}[]

\newtheorem{cor}[theorem]{Corollary}

\newtheorem{lemma}[theorem]{Lemma}

\newtheorem{prop}[theorem]{Proposition}

\theoremstyle{definition}
\newtheorem{define}[theorem]{Definition}
\newtheorem{define*}[theorem*]{Definition}
\newtheorem{ex}[theorem]{Example}
\newtheorem{remark}[theorem]{Remark}

\DeclareFontFamily{OT1}{pzc}{}
\DeclareFontShape{OT1}{pzc}{m}{it}{<-> s * [1.20] pzcmi7t}{}
\DeclareMathAlphabet{\mathpzc}{OT1}{pzc}{m}{it}

\newcommand{\bor}{\textnormal{bor}}
\newcommand{\Hom}{\textnormal{Hom}}
\newcommand{\End}{\textnormal{End}}
\newcommand{\Dom}{\textnormal{Dom}\,}
\newcommand{\K}{\mathbb{K}}
\newcommand{\B}{\mathbb{B}}
\newcommand{\Lip}{\textnormal{Lip}}
\newcommand{\minten}{\overline{\otimes}}
\newcommand{\alg}{\textnormal{alg}}
\newcommand{\N}{\field{N}}
\newcommand{\id}{\mathrm{id}}

\newcommand{\ind}{\mathrm{ind}}

\newcommand{\field}[1]{\mathbb{#1}}

\newcommand{\C}{\field{C}}
\newcommand{\R}{\field{R}}
\newcommand{\Z}{\mathbb{Z}}

\newcommand{\dmin}{\partial^{\textnormal{min}}}
\newcommand{\dmax}{\partial^{\textnormal{max}}}
\newcommand{\Spin}{\textnormal{Spin}}
\newcommand{\cyl}{\textnormal{cyl}}
\newcommand{\amp}{\textnormal{amp}}
\newcommand{\A}{\mathcal{A}}
\newcommand{\der}{\textnormal{d}}
\newcommand{\E}{\mathcal{E}}

\newcommand{\Her}{\textnormal{Her}}
\newcommand{\dc}{\textnormal{dc}}
\newcommand{\un}{\textnormal{un}}
\newcommand{\an}{\textnormal{an}}
\newcommand{\geo}{\textnormal{geo}}
\newcommand{\mini}{\textnormal{min}}
\newcommand{\hc}{\textnormal{hc}}
\newcommand{\s}{\textnormal{s}}

\begin{document}
\title[$KK$-bordisms]{The bordism group of unbounded KK-cycles}

\date{\today}

\author[Deeley]{Robin J. Deeley}
\author[Goffeng]{Magnus Goffeng}
\author[Mesland]{Bram Mesland}

\begin{abstract}
We consider Hilsum's notion of bordism as an equivalence relation on unbounded $KK$-cycles and study the equivalence classes. Upon fixing two $C^*$-algebras, and a $*$-subalgebra dense in the first $C^*$-algebra, a $\Z/2\Z$-graded abelian group is obtained; it maps to the Kasparov $KK$-group of the two $C^*$-algebras via the bounded transform. We study properties of this map both in general and in specific examples. In particular, it is an isomorphism if the first $C^*$-algebra is the complex numbers (i.e., for $K$-theory) and is a split surjection if the first $C^*$-algebra is the continuous functions on a compact manifold with boundary when one uses the Lipschitz functions as the dense $*$-subalgebra.
\end{abstract}

\maketitle

\section*{Introduction}
Kasparov's bivariant $K$-theory (i.e., $KK$-theory) \cite{Kas} is by now a fundamental tool in operator algebras and its applications. At the level of the groups themselves, $KK$-theory provides a joint generalization of $K$-theory and $K$-homology; while at the cycle level, it provides a vast generalization of $*$-homomorphisms, see any of \cite{Bla,cuntzquasi,cumero, Kas}. A prototypical example of a $KK$-cycle is the one associated to an elliptic differential operator acting on the sections of a vector bundle over a smooth manifold. Geometric examples of such classes are obtained from the Dirac operator on a ${\rm spin^c}$-manifold, the signature operator on an oriented manifold, the de Rham operator, among many others.

As these examples show, many $KK$-cycles appear naturally from unbounded operators. This observation led to the framework provided in \cite{baajjulg}, where Baaj and Julg gave the definition of an unbounded $KK$-cycle and defined a map (called the bounded transform) from the  unbounded cycles to the bounded cycles. It was shown in \cite{baajjulg} that every $KK$-class can be represented by an unbounded cycle (i.e., the bounded transform is surjective). After the introduction of the unbounded model in \cite{baajjulg}, Kucerovsky \cite{Kucerovsky1} expanded the theory by placing the Kasparov product in the unbounded setting. Subsequently, Hilsum developed a notion of bordism in the unbounded picture of $KK$-theory in the series of papers \cite{hilsumcmodbun,hilsumfol, hilsumbordism}, motivated by geometric examples from index theory. Further technical advances that enabled a more constructive approach to the unbounded model were made by Kaad and Lesch \cite{leschkaad2}, who independently proved and employed a local-global principle first proved by Pierrot in \cite{Pierrot}, to deal with sums of self-adjoint regular operators in Hilbert $C^{*}$-modules.

In view of these developments, we will in this paper address the following question posed in \cite[Section 17.11]{Bla}:

 \vspace{0.20cm} 
\noindent
``We leave to the reader the task of appropriately formulating the equivalence relations on unbounded cycles corresponding to the standard relations on bounded cycles."
 \vspace{0.20cm}
 
We will interpret the task of Blackadar as the question: is there a relation on unbounded $KK$-cycles, {\bf not} making reference to the bounded model,  which is equivalent to the relation of homotopy of their bounded transforms? In other words, given $C^*$-algebras, $A$ and $B$, can $KK_*(A,B)$ be realized using unbounded cycles and a relation defined at the level of these cycles? The work of Baaj-Julg, as well as the dictionary by Kucerovsky \cite[Appendix]{valettebook} indicate that our way of interpreting the question of Blackadar could have a positive answer. In the present paper we propose a relation at the level of unbounded cycles, with properties similar to Kasparov's notion of homotopy, by combining the technical results of Kaad-Lesch with the notion of Hilsum bordism.

Let us rephrase the question of Blackadar as a precise mathematical question. It requires a bit of notation to state. Let $A$ and $B$ be $C^*$-algebras and $\mathcal{A} \subseteq A$ be a dense $*$-subalgebra; fixing such a choice is similar to fixing a smooth structure. One of the reasons to fix a dense subalgebra of $A$ is due to a technical issue in the unbounded model: the direct sum is in general not well-defined, see more in Appendix \ref{counterdensesub}. An unbounded cycle with respect to $KK_*(A,B)$ is the following data, see Section \ref{unboundedSection} for further details:

\begin{define*} 
(also see Definition \ref{cycleschaindeef}) \\
Let $A$ and $B$ be $C^*$-algebras. An unbounded $KK$-cycle (with respect to $(A,B)$) is a pair $(\mathpzc{E}, D)$ where
\begin{enumerate}
\item $\mathpzc{E}$ is a (graded) Hilbert $(A,B)$-bimodule;
\item $D$ is an (odd) self-adjoint regular operator acting as an unbounded operator on $\mathpzc{E}$;
\item For each $a \in A$, $a(D\pm i)^{-1} \in \K_B(\mathpzc{E})$; 
\item The Lipschitz algebra $\mathrm{Lip}(A,\mathpzc{E},D)$ is dense in $A$, where:
\begin{align*}
\mathrm{Lip}&(A,\mathpzc{E},D)\\
& := \{ a\in A \: : \: a\cdot \Dom(D) \subseteq \Dom(D) \text{ and } [D,a] \in \End_B^{*}(\mathpzc{E}) \}.
\end{align*}
\end{enumerate}
\end{define*}

Here we abuse notation by writing $[D,a] \in \End^{*}_B(\mathpzc{E})$ if $[D,a]$ has a bounded adjointable extension to $\mathpzc{E}$. We write $\Lip(D)$ for the case $A=\End^{*}_{B}(\mathpzc{E})$, cf. Definition \ref{Lip}. For two unbounded cycles $(\mathpzc{E},D)$ and $(\mathpzc{E}',D')$, it holds that $\mathrm{Lip}(D\oplus D')=\mathrm{Lip}(D)\cap \mathrm{Lip}(D')$. Therefore, if both Lipschitz algebras contain $\mathcal{A}$, then the direct sum $(\mathpzc{E}\oplus \mathpzc{E}',D\oplus D')$ is also an unbounded cycle whose Lipschitz algebra contains $\mathcal{A}$.

{ \bf Question:} Does there exist an additive equivalence relation $\sim_{\textnormal{un}}$ on unbounded cycles such that given $C^*$-algebras, $A$ and $B$, there exists a dense subalgebra $\mathcal{A}\subseteq A$ such that the bounded transform: 
\[
\left\{ (\mathpzc{E}, D):\; \mathcal{A}\subseteq \mathrm{Lip}(D) \right\}/\!\! \sim_{\un}\;   \longrightarrow  KK_*(A,B)
\]
 is an isomorphism of abelian groups?\newline

The reader familiar with the Baum-Douglas model for $K$-homology (see \cite{BD}) might find the following analogy useful. The Baum-Douglas model uses geometric cycles, $(M,E,f)$ (see Definition \ref{BDcycle}), to give a realization of $KK_*(C(X), \C)$ where $X$ is a finite CW-complex. In particular, given a cycle $(M,E,f)$ there is an associated cycle in $KK$-theory; it is denoted by $f_*([D_E])$. We denote the mapping induced from $(M,E,f)\mapsto f_*([D_E])$ by $\mu$. Thus, every geometric cycle gives a class in $KK_*(C(X),\C)$. This is analogous to the bounded transform taking an unbounded cycle to a bounded one.

Moreover, a discussion similar to the one above (regarding unbounded $KK$-theory in the context of this geometric model) takes the following form: is there an equivalence relation on the geometric cycles which turns $\mu$ into an isomorphism? Of course, there is such a relation, see \cite{BD}. Indeed, the fact that the relation on Baum-Douglas cycles is geometrically defined is just as important as the cycles themselves being geometric.

The construction of an equivalence relation on unbounded $KK$-cycles is the theoretical starting point for this paper. The technical starting point is Hilsum's notion of bordism in the context of unbounded $KK$-theory--in particular, \cite[Theorem 6.2]{hilsumbordism}. The notion of bordism due to Hilsum is not to be confused with that of cobordism for {\it bounded} $KK$-cycles \cite{CS} (also see \cite[Section 17.10]{Bla}). {\bf Cobordism for bounded $KK$-cycles is  not the relation considered here.} Hilsum uses his notion of bordism to prove results concerning the cobordism invariance of various indices, see for examples \cite[Theorem 8.4 and Corollary 9.3]{hilsumbordism}. In particular, \cite[Theorem 6.2]{hilsumbordism} implies that if there is a Hilsum bordism between two unbounded $KK$-cycles, then the classes of the associated bounded transforms are equal in the relevant $KK$-group.

This result naturally led us to consider the possibility of defining an equivalence relation $\sim_{\bor}$ using Hilsum bordisms--that one can define such a relation is one of the fundamental results of the present paper, see Proposition \ref{bordismequi} and Definition \ref{defOfBorGrp}. The equivalence classes under $\sim_{\bor}$ form an abelian group. We denote the group associated to $\mathcal{A} \subseteq A$ and $B$ (see Definition \ref{defOfBorGrp}) by  $\Omega_*(\mathcal{A},B)$. Furthermore, using \cite[Theorem 6.2]{hilsumbordism}, one obtains the following result, which appears as Theorem \ref{bddtrasnsurj} below:

\begin{theorem*}
The abelian semigroup $\Omega_*(\mathcal{A},B)$ forms a $\Z/2\Z$-graded abelian group and the bounded transform 
$$b:\Omega_*(\mathcal{A},B)\to KK_*(A,B), \quad b[\mathpzc{E},D]:=[\mathpzc{E},b(D)],$$
is a well defined group homomorphism. Moreover, given two separable $C^*$-algebras $A$ and $B$, there exists a dense $*$-subalgebra $\mathcal{A}\subseteq A$ such that $b:\Omega_*(\mathcal{A},B)\to KK_*(A,B)$ is surjective. If $KK_*(A,B)$ is countably generated, $\mathcal{A}$ can be taken to be a Fr\'{e}chet algebra and if $KK_*(A,B)$ is finitely generated, $\mathcal{A}$ can be taken to be a Banach algebra.
\end{theorem*}

As mentioned above, the Baum-Douglas model also provides a realization of certain $KK$-groups. In fact, if $X$ is a compact manifold with boundary and $B$ is a unital $C^*$-algebra, one can model $KK_*(C(X), B)$ using Baum-Douglas type cycles, see for example \cite{Wal} or the discussion in Section \ref{BDBorKKSection}. The relationship between the geometric group, the $KK$-bordism group, and the standard Kasparov group is summarized in the next theorem, which appears as Theorem \ref{betamugamma} and Corollary \ref{splittingbeta} below--further details are provided in Subsection \ref{mfdwbdryex} and Section \ref{BDBorKKSection}. The definition of $\gamma_0$ can be found in Lemma \ref{firstgammadef} (on page \pageref{firstgammadef}). 

\begin{theorem*}
If $X$ is a compact manifold with boundary, the mapping 
$$\gamma:K^{{\rm geo}}_{*}(X;B) \to \Omega_*(\Lip(X),B)),\quad \gamma(M,E_B,\nabla_E,f):=f_*\gamma_0(M,E_B,\nabla_E),$$
is a well defined split injection and fits into a commuting diagram:
\begin{equation}
\xymatrix@C=1.7em@R=1.2em{
 K^{{\rm geo}}_{*}(X;B) \ar[rdd]_{\mu}   \ar[rr]^{\gamma}  & & \Omega_*(\Lip(X),B))  \ar[ldd]^{b}  
 \\ \\
&KK_*(C(X),B)&
}
\end{equation}
Moreover, the bounded transform $b: \Omega_*(\Lip(X),B))\to KK_*(C(X),B)$ is a split surjection.
\end{theorem*}

The main technical tool in the proof of the above theorem is the notion of weakly degenerate cycles, see Definition \ref{weakdegdeg}, and the fact that these cycles are nullbordant, see Theorem \ref{weakdegnullbor}. In one special case, we do show that the bordism group is isomorphic to the Kasparov group. We have the following theorem, which appears as Theorem \ref{bddtrandk} and Corollary \ref{symnormediso}:

\begin{theorem*}
The bounded transform $b:\Omega_*(\C,B)\to KK_*(\C,B)$ is an isomorphism. Furthermore, if $\mathcal{I}$ is a regular symmetrically normed ideal of compact operators $b:\Omega_*^e(\mathcal{I},B)\to KK_*(\K,B)$ is an isomorphism, here $\Omega_*^e(\mathcal{I},B)$ denotes the subgroup generated by essential cycles (see Definition \ref{esssubgroup}).
\end{theorem*}

There are at least two natural questions that we do not address, but are of great interest for further development of the theory:
\begin{enumerate}
\item A detailed study of the dependence on the bordism group on the dense subalgebra used to define it;
\item The definition of products in the bordism groups.
\end{enumerate}
In fact, we expect these questions to be related. It is likely that one must fix particular properties on the dense subalgebra, as well as restrict to a class of well behaved cycles, to ensure a well-defined product.

The content of the paper is organized as follows. The first section contains the fundamental properties of unbounded $KK$-cycles, including a review of regularity for unbounded operators on Hilbert modules and products in various contexts (e.g., the exterior product for symmetric chains and products with $K$-theory classes). In Section \ref{bordismsection}, we move to the fundamental object of study: the $KK$-bordism group. The bounded transform and the relationship between the $KK$-bordism group and $KK$-theory is discussed in Sections \ref{boundSection} and \ref{BDBorKKSection}. In the case of the latter section, the Baum-Douglas model for $K$-homology plays a key r\^ole. The final section of the paper contains examples, both of Hilsum bordisms and bordism groups. The reader unfamiliar with Hilsum's notion of bordism might find Subsection \ref{exHilsumSubSec} useful as it provides a list of geometrically defined examples due to Hilsum; this subsection is written to be independent of the rest of the paper.

\section{Unbounded $KK$-cycles}
\label{unboundedSection}

The following notation is used throughout the paper. We use $A$ and $B$ to denote separable $C^*$-algebras. We fix a dense $*$-subalgebra $\mathcal{A}\subseteq A$ throughout. We also assume that the dense subalgebra $\A$ has a locally convex topology stronger than the $C^*$-topology. This is no severe restriction since we can always equip it with the fine topology. If $\A$ has a countable basis, \cite{cuntzweyl} shows that the fine topology makes $\A$ into a locally convex algebra. In examples, $\A$ is a Fr\'echet- or Banach algebra. The topology is in the bulk of the paper \emph{only} used to define smooth functions into $\mathcal{A}$, but as the example in Subsection \ref{compsubse} shows the topology plays a r\^{o}le in computing the bordism groups. A $*$-homomorphism $\mathcal{A}\to \mathcal{A}'$ between two such subalgebras is tacitly assumed to be continuous, also in the $C^*$-topologies, thus being uniquely determined by a $*$-homomorphism $A\to A'$. The $C^*$-continuity is automatic for any $*$-homomorphism if $\mathcal{A}$ is closed under holomorphic functional calculus.

Hilbert $C^*$-modules will be denoted by $\mathpzc{E}$ and $\mathpzc{F}$. The $C^*$-algebra of $B$-linear adjointable operators on the $B$-Hilbert $C^*$-module $\mathpzc{E}$ is denoted by $\End^*_B(\mathpzc{E})$ and $\K_B(\mathpzc{E})\subseteq \End^*_B(\mathpzc{E})$ denotes the $C^*$-algebra of $B$-compact operators. An $(A,B)$-Hilbert bimodule is a $B$-Hilbert module $\mathpzc{E}$ equipped with a $*$-homomorphism $A\to \End_B^*(\mathpzc{E})$. Gradings will be denoted by $\gamma_\mathpzc{E}$ or simply $\gamma$. An $(A,B)$-Hilbert $C^*$-module $\mathpzc{E}$ is said to be graded if $\mathpzc{E}$ is graded as a $B$-Hilbert $C^*$-module and $A$ acts by even operators. Modules, and bimodules, are (possibly trivially) graded unless otherwise stated. The internal tensor product of a Hilbert $B$-module $\mathpzc{E}$ with a $(B,C)$-bimodule $\mathpzc{F}$ is denoted $\mathpzc{E}\otimes_{B}\mathpzc{F}$. If $\mathpzc{F}$ is an $(A,B)$-Hilbert $C^*$-module and $\mathpzc{E}$ is a $B$-Hilbert $C^*$-module we say that $T\in \Hom_B^*(\mathpzc{E},\mathpzc{F})$ is \emph{locally compact} if for any $a\in A$ it holds that $aT\in \K_B^*(\mathpzc{E},\mathpzc{F})$.

\subsection{Unbounded operators and regularity}

Let us first recall some notions for unbounded operators. The main reference for this subsection is \cite[Chapter 9 and 10]{lancesbook}. An unbounded operator $D:\mathpzc{E}\dashrightarrow \mathpzc{F}$ is a $B$-linear operator $D:\Dom(D)\to \mathpzc{F}$ for a $B$-submodule $\Dom(D)\subseteq \mathpzc{E}$. If $\Dom(D)$ is dense, we say that $D$ is densely defined. If $D:\mathpzc{E}\dashrightarrow\mathpzc{E}$ for a graded $C^*$-module $\mathpzc{E}$, we say $D$ is odd if $\gamma_\mathpzc{E}$ preserves $\mathrm{Dom}(D)$ and $D\gamma_\mathpzc{E}=- \gamma_\mathpzc{E}D$. If $D\gamma_\mathpzc{E}= \gamma_\mathpzc{E}D$ , we say that $D$ is even. The operator $D$ is called closed if its graph
$$G(D):=\{(x,Dx):\; x\in \Dom(D)\}\subseteq \mathpzc{E}\oplus \mathpzc{F},$$
is a closed subspace.
With a closed operator $D$, we associate the $B$-Hilbert $C^*$-module $W(D):=\Dom(D)$ equipped with the $B$-valued inner product
\begin{equation}
\label{dominprodef}
\langle x,y\rangle_{W(D)}:=\langle x,y\rangle_\mathpzc{E}+\langle Dx,Dy\rangle_\mathpzc{F}.
\end{equation}

The adjoint of a densely defined unbounded operator is an unbounded operator $D^*:\mathpzc{F}\dashrightarrow\mathpzc{E}$ equipped with the domain 
$$\Dom(D^*):=\{y\in \mathpzc{F}: \; \exists z\in \mathpzc{E}, \; \langle Dx,y\rangle_\mathpzc{F}=\langle x,z\rangle_\mathpzc{E} \;\forall x\in \Dom(D)\}.$$
Then $D^*$ is defined by $D^*y:=z$. The adjoint is always closed because 
$$G(D^*)=vG(D)^\perp, \quad\mbox{where}\quad v=\begin{pmatrix} 0& 1\\ -1& 0\end{pmatrix}.$$
If $D^*=D$ we say that $D$ is self-adjoint and if $D\subseteq D^*$ we say that $D$ is symmetric. Symmetric operators are closable, and we will thus assume all symmetric operators to be closed. By \cite[Lemma 9.7]{lancesbook}, if $D$ is closed and symmetric then $D\pm i$ are injective with closed range.

Following \cite{leschkaad2,posepumpen} we say that a closed operator $D$ is semi-regular if both $D$ and $D^*$ are densely defined. {\bf We henceforth assume that all our operators are semi-regular.} The next theorem is related to the highly relevant notion of \emph{regularity} for unbounded operators. For its proof, we refer to \cite{leschkaad2,lancesbook}.

\begin{theorem}
\label{regproperties}
Let $D$ be a semi-regular operator. The following are equivalent:
\begin{enumerate}
\item $1+D^*D$ has dense range.
\item $1+D^*D$ is surjective.
\item $G(D)$ is a complemented $B$-submodule in $\mathpzc{E}\oplus \mathpzc{F}$.
\item The canonical mapping $\mathpzc{E}\oplus \mathpzc{F}\cong G(D)\oplus vG(D^*)$ is an isomorphism.
\item $D$ defines a bounded adjointable mapping $D:W(D)\to \mathcal{F}$.
\item When localizing $D$ and $D^*$ in the GNS-space $\mathpzc{H}_\omega$ of a state $\omega$ on $B$, it holds that $(D_\omega)^*=(D^*)_\omega$ as operators on the Hilbert space $\mathpzc{E}_\omega:=\mathpzc{E}\otimes_B\mathpzc{H}_\omega$.
\end{enumerate}
\end{theorem}

The last condition of Theorem \ref{regproperties} is known as the local-global principle. This principle first appeard in \cite[Theorem 1.18]{Pierrot}. The formulation in Theorem \ref{regproperties} is from \cite[Theorem 4.2]{leschkaad2}. Since it reduces to a property on Hilbert spaces, it plays an important technical r\^ole throughout the paper. If a semi-regular operator $D$ satisfies any of the conditions of Theorem \ref{regproperties}, we say that $D$ is regular. It is sometimes possible to reduce regularity to a property of a self-adjoint operator. For a semi-regular operator $D:\mathpzc{E}\dashrightarrow\mathpzc{E}$ we use the notation 
\begin{equation}
\label{twiddlingoperators}
\tilde{D}:=\begin{pmatrix} 0& D\\ D^*& 0\end{pmatrix}.
\end{equation}
We refer to $\tilde{D}$ as the double of $D$. This operator is semi-regular and symmetric with $\Dom(\tilde{D})=\Dom(D^*)\oplus \Dom(D)$. It follows by \cite[Lemma 2.3]{leschkaad2} that $D$ is regular if and only if $\tilde{D}$ is regular and self-adjoint.

\begin{prop}
\label{regulproperla}
Let $D$ be a semi-regular operator on a Hilbert $C^{*}$-module $\mathpzc{E}$. Then
\begin{enumerate}
\item If $D$ is regular, then so is $D^*$.
\item If $D$ is regular, then $D^{**}=D$.
\item For a regular operator $D$ the operator $D^*D$ is regular, self-adjoint and $\Dom(D^*D)$ is a core for $D$.
\item If $D$ is symmetric, $D$ is regular if and only if $D\pm i$ have complemented ranges.
\item If $D$ is symmetric, then $D$ is self-adjoint and regular if and only if $D\pm i:\Dom D\to \mathpzc{E} $ have dense range if and only if $D\pm i$ are surjective.
\end{enumerate}
\end{prop}

The proofs of the properties in Proposition \ref{regulproperla} can be found in \cite[Chapter 9, 10]{lancesbook}. A result of technical importance to this paper is the following generalization of \cite[Theorem 7.10]{leschkaad2}. It concerns regularity of the sum of a symmetric and a self-adjoint operator.

\begin{theorem}
\label{sumregul}
Suppose that $S$ is a self-adjoint regular operator and $T$ is a symmetric regular operator on a Hilbert $C^*$-module $\mathpzc{E}$. We assume for any $\mu\in \R\setminus \{0\}$ that
\begin{enumerate}
\item $(S+i\mu)^{-1}\Dom(T)\subseteq \Dom(S)\cap \Dom(T)\cap \Dom(ST)$ and the operator $(ST+TS)(S+i\mu)^{-1}$ has a bounded adjointable extension to $\mathpzc{E}$;
\item $(S+i\mu)^{-1}\Dom(T^*)\subseteq \Dom(S)\cap \Dom(T^*)\cap \Dom(ST^*)$ and the operator $(ST^*+T^{*}S)(S+i\mu)^{-1}$ has a bounded adjointable extension to $\mathpzc{E}$.
\end{enumerate}
Then $S+T$ is a closed regular symmetric operator with domain $\Dom(S)\cap \Dom(T)$ and $(S+T)^*$ is the closure of $S+T^*$ equipped with the domain $\Dom(S)\cap \Dom(T^*)$.
\end{theorem}

\begin{proof} 
The proof will follow by a careful inspection of the proofs of \cite[Section 7]{leschkaad2}. The difference is that in the statements in \cite{leschkaad2}, for instance \cite[Proposition 7.7]{leschkaad2}, it is required that $T$ is self-adjoint; the proofs can be generalized to symmetric operators $T$. We will indicate where $T$ and $T^*$ are used and where the conditions (1) and (2) appear. 

Condition (1) \emph{and} the fact that $T$ is symmetric imply that $S+T$ is a closed operator on $\Dom(S)\cap\Dom(T)$ using the argument in \cite[Lemma 7.6]{leschkaad2}. A quick computation shows that $$\Dom S\cap\Dom T^{*}\subset \Dom (S+T)^{*},$$ which implies that $(S+T)^{*}$ is densely defined, so $S+T$ is semi-regular. Furthermore, the graph of 
\[S+T^{*}:\Dom S\cap \Dom T^{*}\to \mathpzc{E},\]
is contained in the graph of the closed operator $(S+T)^{*}$, so $S+T^{*}$ is closable. We denote its closure by $\overline{S+T^*}$. Likewise, $(S+T^*)^*$ is an extension of the operator $S+T$. 

Let us prove the identity $S+T=(S+T^*)^*$. Condition (2) implies that for any $\mu\in \R\setminus \{0\}$ on $\Dom(T^*)$:
\begin{equation}
\label{resolventcomputanticomm}
(i\mu S+1)^{-1}T^*+T^*(i\mu S-1)^{-1}=i\mu(i\mu S+1)^{-1}(ST^*+T^*S)(i\mu S-1)^{-1}.
\end{equation}
The right hand side of this equation has a bounded adjointable extension $R_\mu$ by Condition (2). An argument as in \cite[Lemma 7.4]{leschkaad2} shows that the $R_\mu$ converge strictly to $0$ as $\mu\to0$. Following the proof of \cite[Proposition 7.7]{leschkaad2}, it follows from Equation \eqref{resolventcomputanticomm} that for any $\eta\in \Dom(T^*)$ and $\xi\in \Dom((S+T^*)^*)$ 
\begin{align}
\label{resolvcomp}
\left\langle \left(-\frac{i}{n}S+1\right)^{-1}\xi, T^*\eta\right\rangle=&\left\langle \left(\frac{i}{n}S+1\right)^{-1}(S+T^*)^*\xi,\eta\right\rangle\\
\nonumber
&-\left\langle Sc\left(\frac{S}{n}\right)\left(-\frac{i}{n}S+1\right)^{-1}\xi, \eta\right\rangle+\left\langle R_{\frac{1}{n}}^*\xi,\eta\right\rangle,
\end{align}
where $c(s):=(1-is)(1+is)^{-1}$. It follows from \cite[Lemma 7.2]{leschkaad2} that $c(S/n)\to 1$ and $(\frac{i}{n}S+1)^{-1}\to 1$ strictly. Consider the sequence $\xi_n:=(-\frac{i}{n}S+1)^{-1}\xi$. From \eqref{resolvcomp}, we conclude that $\xi_n\in\Dom(S)\cap\Dom(T)$ for all $n$ and that $\xi_n\to \xi$ in the graph norm of $\Dom((S+T^*)^*)$. Therefore $\Dom(S)\cap\Dom(T)\subseteq  \Dom((S+T^*)^*)$ is dense in the graph norm and $S+T=(S+T^*)^*$.

The result will now follow from the localization techniques of \cite{leschkaad2,Pierrot}. In more detail, since $S+T$ is semi-regular, and in particular closed, the proof of the result is completed using the local-global principle upon proving that $(S_{\omega}+T_{\omega})^*=((S+T)^{*})_{\omega}$ in every localization $\mathpzc{E}_{\omega}$ of $\mathpzc{E}$ in a state $\omega$ on $B$  (see \cite[Theorem 1.18]{Pierrot},  \cite[Theorem 4.2]{leschkaad2}). We remark that the operators $S_\omega$ and $T_\omega$ also satisfy Condition (1) and (2), this is shown in \cite[proof of Lemma 7.9]{leschkaad2}. The equality $(S+T)_{\omega}=S_{\omega}+T_{\omega}$ follows from the fact that the operators are defined on a common core and that $S_{\omega}+T_{\omega}$ is closed by Condition (1). By a similar argument, the identity $(\overline{S+T^{*}})_{\omega}=\overline{S_{\omega}+T^{*}_{\omega}}$ follows. As the operators $S_\omega$ and $T_\omega$ also satisfy Condition (1) and (2), $(S_\omega+T_\omega)^*=\overline{S_\omega+T^*_\omega}$, and hence $(S_{\omega}+T_{\omega})^*=((S+T)^{*})_{\omega}$.
\end{proof}

\begin{define}
\label{Lip}
Let $D:\mathpzc{E}\dashrightarrow \mathpzc{E}$ be a semi-regular operator. The algebra of $D$-\emph{Lipschitz operators} $\Lip(D)\subseteq \End^*_B(\mathpzc{E})$ is defined as the algebra of bounded adjointable operators $T$ such that 
\begin{enumerate}
\item $T$ preserves $\mathrm{Dom}(D)$.
\item $[D,T]:\mathrm{Dom}(D)\to \mathpzc{E}$ extends to a bounded adjointable operator $\mathpzc{E}\to \mathpzc{E}$. 
\end{enumerate}
We always consider $\Lip(D)$ as a Banach algebra in the norm 
$$\|T\|_{\Lip(D)}:=\|T\|_{\End^*_B(\mathpzc{E})}+\|[D,T]\|_{\End^*_B(\mathpzc{E})}.$$

If $D$ is a symmetric operator, we define the algebra of $(D^*,D)$-Lipschitz operators $\Lip(D^*,D)$ as the left ideal in $\Lip(D)$ consisting of operators $T\in \Lip(D)$ such that 
$$T\mathrm{Dom}(D^*)\subseteq \Dom(D).$$  
\end{define}
\begin{remark}\label{LipDstar} If $a,a^{*}\in\Lip(D)$, then it follows from \cite[Proposition 2.1]{FMR} that $a,a^{*}\in\Lip(D^{*})$. In particular, there is a continuous inclusion $\Lip(D)\cap \Lip(D)^*\subseteq \Lip(D^{*})$.
\end{remark}

\subsection{Chains and cycles in unbounded $KK$-theory}

Throughout the paper, we will consider pairs $(\mathpzc{E},D)$ of a Hilbert $C^*$-module and a regular operator. The following definition provides us with the terminology for the various conditions that will be placed upon such pairs.

\begin{define}[cf. \cite{hilsumbordism}]
\label{cycleschaindeef}
An $(\mathcal{A},B)$-\emph{chain} is a pair $(\mathpzc{E}, D)$ consisting of an $(A,B)$-Hilbert $C^*$-module $\mathpzc{E}$ and $D$ a regular operator (odd if $\mathpzc{E}$ is non-trivially graded) such that the $*$-representation $\pi:A\to \End^*_B(\mathpzc{E})$ restricts to a continuous homomorphism:
\[\pi:\mathcal{A}\to \Lip(D). \]
The $(\mathcal{A},B)$-chain $(\mathpzc{E}, D)$ is said to be \emph{symmetric} if $D$ is symmetric. A symmetric $(\mathcal{A},B)$-chain $(\mathpzc{E}, D)$ is said to be \emph{half-closed} if 
$$\pi(\mathcal{A})\subseteq \Lip(D^*,D)$$ 
and, for any $a\in A$,
\begin{equation}
\label{locacomp}
\pi(a)(1+D^*D)^{-1}\in \K_B(\mathpzc{E}).
\end{equation}
The $(\mathcal{A},B)$-chain $(\mathpzc{E}, D)$ is said to be \emph{closed} if $D= D^*$ and, for any $a\in A$,
$$\pi(a)(i\pm D)^{-1}\in \K_B(\mathpzc{E}).$$
\end{define}

An \emph{isomorphism} of two $(\mathcal{A},B)$-chains $(\mathpzc{E}, D)$ and $(\mathpzc{E}', D')$ is a unitary isomorphism $u:\mathpzc{E}\to \mathpzc{E}'$ of $(A,B)$-Hilbert $C^*$-modules such that $u\Dom(D)=\Dom(D')$ and $D=uD'u^*$. The conditions closed and half-closed are invariant under isomorphism. 

We refer to a closed chain as a cycle. We will often, but not exclusively, denote cycles by $(\mathpzc{E},D)$ and symmetric chains by $(\mathpzc{F},Q)$. An operator $D$ satisfying the condition in Equation \eqref{locacomp} is said to have \emph{$A$-locally compact resolvent}. If $A$ is understood from context, we simply say that $D$ has locally compact resolvent.

\begin{define}
\label{cycleschaindeefnew}
If $\mathpzc{E}$ is trivially graded, the chain is odd, otherwise we call it even. The dimension modulo $2$ of $(\mathpzc{E}, D)$ is denoted by $\dim_{\Z/2}(\mathpzc{E}, D)$ and defined as follows: $\dim_{\Z/2}(\mathpzc{E}, D):=0\mod 2$ if $(\mathpzc{E}, D)$ is even and $\dim_{\Z/2}(\mathpzc{E}, D):=1\mod 2$ if $(\mathpzc{E}, D)$ is odd.
\end{define}

\begin{define}
The \emph{inverse} of an $(\mathcal{A},B)$-chain $(\mathpzc{E}, D)$ is given by
$$-(\mathpzc{E}, D)=
\begin{cases}
(-\mathpzc{E}, D), \quad\mbox{for even $(\mathpzc{E}, D)$},\\
(\mathpzc{E}, -D), \quad\mbox{for odd $(\mathpzc{E}, D)$.}
\end{cases}$$
Here $-\mathpzc{E}$ is the module $\mathpzc{E}$ with the opposite grading. 

The \emph{sum} of two $(\mathcal{A},B)$-chains $(\mathpzc{E}_1, D_1)$ and $(\mathpzc{E}_2, D_2)$ is defined by
\begin{equation}
\label{dirsum}
(\mathpzc{E}_1, D_1)+(\mathpzc{E}_2, D_2):=(\mathpzc{E}_1\oplus \mathpzc{E}_2, D_1\oplus D_2).
\end{equation}
\end{define}

It is clear from the properties of regular operators that the inverse of a chain and the sum of two chains are again chains. We introduce the notation $C_*^{\s}(\mathcal{A},B)$ for the set of isomorphism classes of symmetric chains, this is a set $\Z/2$-graded by the dimension of the cycle. Similarly, we let $C^{\hc}_*(\mathcal{A},B)$ and $Z_*(\mathcal{A},B)$ denote the set of isomorphism classes of half-closed respectively closed chains, again these sets are $\Z/2$-graded by dimension. 

\begin{prop}
\label{semiisoclass}
The sets $C_*^{\s}(\mathcal{A},B)$, $C^{\hc}_*(\mathcal{A},B)$ and $Z_*(\mathcal{A},B)$ form $\Z/2\Z$-graded abelian semigroups under the direct sum operation in Equation \eqref{dirsum} and the inclusions $C_*^{\s}(\mathcal{A},B)\supseteq C^{\hc}_*(\mathcal{A},B)\supseteq Z_*(\mathcal{A},B)$ are inclusions of semigroups. The $\Z/2\Z$-graded abelian semigroups $C_*^{\s}(\mathcal{A},B)$, $C^{\hc}_*(\mathcal{A},B)$ and $Z_*(\mathcal{A},B)$ depend contravariantly on $\mathcal{A}$ and covariantly on $B$ with respect to $*$-homomorphisms under the constructions \eqref{concon} and respectively \eqref{covcon} below.

\end{prop}
\begin{proof}
The proof that the semigroups are abelian follows from applying standard techniques with the flip map. For functoriality, let $\alpha:\mathcal{A}_2\to \mathcal{A}_1$ be a $*$-homomorphism. For an $(\mathcal{A}_1,B)$-chain $(\mathpzc{E},D)$ we define the $(\mathcal{A}_2,B)$-chain 
\begin{equation}
\label{concon}
\alpha^*(\mathpzc{E},D):=(\alpha^*\mathpzc{E},D),
\end{equation}
where $\alpha^*\mathpzc{E}:=\mathpzc{E}$ equipped with the left $A_2$-action defined from $\alpha$. If $\alpha':B_1\to B_2$ is a $*$-homomorphism and $(\mathpzc{E},D)$ is an $(\mathcal{A},B_1)$-chain, we define the $(\mathcal{A},B_2)$-chain
\begin{equation}
\label{covcon}
\alpha'_*(\mathpzc{E},D):=(\mathpzc{E}\otimes_{\alpha'}B_2,D\otimes_{\alpha'}1_{B_2}).
\end{equation}
Here $D\otimes_{\alpha'}1_{B_2}$ is the regular operator on $\mathpzc{E}\otimes_{\alpha'}B_2$ given as the internal tensor product with the identity operator on $B_2$, see \cite[Chapter 9]{lancesbook}.
\end{proof}

We now turn to the bounded transform. This transform relates the cycles and even the half-closed chains considered above with $KK$-theory. For a regular operator $D:\mathpzc{E}\dashrightarrow\mathpzc{E}$ we define
$$b(D):=D(1+D^*D)^{-1/2}.$$
This operator is a bounded adjointable operator on $\mathpzc{E}$ by \cite[Chapter 9]{lancesbook}. By \cite[Theorem 3.2]{hilsumbordism}, for any half-closed $(\mathcal{A},B)$-chain $(\mathpzc{E},D)$ the pair $(\mathpzc{E},b(D))$ is an $(A,B)$-Kasparov cycle. Moreover, by \cite[Lemma 3.1]{hilsumbordism}, the class $[\mathpzc{E},b(\hat{D})]\in KK_*(A,B)$ is well defined and independent of closed regular extension $D\subseteq \hat{D}\subseteq D^*$. Here we tacitly assume $\hat{D}$ to be odd if $\dim_{\Z/2}(\mathpzc{E},D)=0$.

\begin{theorem}[cf. \cite{baajjulg,hilsumbordism}]
\label{surjoncycles}
The map
$$b:C^{\hc}_*(\mathcal{A},B)\to KK_*(A,B),\quad (\mathpzc{E},D)\mapsto [\mathpzc{E},b(D)],$$ 
is a well defined and additive. Moreover, given two $C^*$-algebras $A$ and $B$, there exists a complete locally convex $*$-subalgebra $\mathcal{A}\subseteq A$ such that 
$$b|_{Z_*}:Z_*(\mathcal{A},B)\to KK_*(A,B),$$ 
is surjective. If $KK_*(A,B)$ is countably generated, $\mathcal{A}$ can be taken to be a Fr\'echet algebra and if $KK_*(A,B)$ is finitely generated, $\mathcal{A}$ can be taken to be a Banach algebra.
\end{theorem}

\begin{ex}
An important example of a chain is obtained from the Dirac operator on a smooth manifold $W$ with coefficients in a smooth bundle of finitely generated projective modules over a $C^*$-algebra $B$. If $W$ is a closed manifold, \cite[Theorem 2.3]{hankpapsch} shows that we obtain a  cycle for the Lipschitz functions on $W$. If $W$ has a boundary, the Dirac operator with its minimal boundary condition provides a half-closed chain over Lipschitz functions vanishing on the boundary and a symmetric chain for the Lipschitz functions on the whole manifold with boundary. See more below in Subsection \ref{mfdwbdryex}.
\end{ex}

\subsection{Exterior products of symmetric chains}
\label{extprodofsymsect}

One of the original motivations for Baaj-Julg to introduce unbounded $KK$-cycles in \cite{baajjulg} was to explicitly describe the exterior product of cycles. The exterior product can often be extended to symmetric chains, a technicality we will need later. This product was considered by Hilsum in \cite[Section 3.1]{hilsumbordism}. 

Given a $B_1$-Hilbert $C^*$-module $\mathpzc{E}_1$ and a $B_2$-Hilbert $C^*$-module $\mathpzc{E}_2$, we denote their exterior tensor product by $\mathpzc{E}_1\boxtimes \mathpzc{E}_2$ which is a $B_1\minten B_2$-Hilbert $C^*$-module, where $B_1\minten B_2$ is the minimal tensor product of $C^{*}$-algebras. The Hilbert $C^*$-modules can be graded, and we define their graded exterior tensor product by
$$\mathpzc{E}_1\hat{\boxtimes} \mathpzc{E}_2:=
\begin{cases}
\mathpzc{E}_1\boxtimes \mathpzc{E}_2,\quad \mbox{if at least one of the modules is nontrivially graded},\\
\mathpzc{E}_1\boxtimes \mathpzc{E}_2\oplus \mathpzc{E}_1\boxtimes \mathpzc{E}_2,\quad \mbox{if both modules are trivially graded}.
\end{cases}$$
If both $\mathpzc{E}_1$ and $\mathpzc{E}_2$ are graded, we grade $\mathpzc{E}_1\hat{\boxtimes} \mathpzc{E}_2$ using the product grading. If exactly one of $\mathpzc{E}_1$ or $\mathpzc{E}_2$ is graded, we view $\mathpzc{E}_1\hat{\boxtimes} \mathpzc{E}_2$ as an ungraded Hilbert $C^*$-module. If both $\mathpzc{E}_1$ and $\mathpzc{E}_2$ are ungraded, $\mathpzc{E}_1\hat{\boxtimes} \mathpzc{E}_2$ is graded as a direct sum. 

Let $D_1:\mathpzc{E}_1\dashrightarrow \mathpzc{E}_1$ and $D_2:\mathpzc{E}_2\dashrightarrow \mathpzc{E}_2$ be densely defined regular operators. For any $T_1\in \End^*_{B_1}(\mathpzc{E}_1)$ and $T_2\in \End^*_{B_2}(\mathpzc{E}_2)$, the operators $D_1\boxtimes T_2:\mathpzc{E}_1\boxtimes\mathpzc{E}_2\dashrightarrow \mathpzc{E}_1\boxtimes\mathpzc{E}_2$ and $T_1\boxtimes D_2:\mathpzc{E}_1\boxtimes\mathpzc{E}_2\dashrightarrow \mathpzc{E}_1\boxtimes\mathpzc{E}_2$ are densely defined regular operators by \cite[Chapter 10]{lancesbook} such that 
\begin{equation}
\label{adjoints}
(D_1\boxtimes T_2)^*=D_1^*\boxtimes T_2^*\quad\mbox{and}\quad (T_1\boxtimes D_2)^*=T_1^*\boxtimes D_2^*.
\end{equation}
We define the graded exterior product of $D_1$ and $D_2$ by
$$D_1\hat{\boxtimes}D_2:=
\begin{cases}
D_1\boxtimes\id_{\mathpzc{E}_2}+\gamma_{\mathpzc{E}_1}\boxtimes D_2,\quad& \mbox{if both $\mathpzc{E}_1$ and $\mathpzc{E}_2$ are graded};\\
i\gamma_{\mathpzc{E}_1}D_1\boxtimes\id_{\mathpzc{E}_2}+\gamma_{\mathpzc{E}_1}\boxtimes D_2, &\mbox{if $\mathpzc{E}_1$ is graded but $\mathpzc{E}_2$ is not};\\
D_1\boxtimes\gamma_{\mathpzc{E}_2}+\id_{\mathpzc{E}_1}\boxtimes i\gamma_{\mathpzc{E}_2}D_2, &\mbox{if $\mathpzc{E}_2$ is graded but $\mathpzc{E}_1$ is not},
\end{cases}$$
and finally, if both $\mathpzc{E}_1$ and $\mathpzc{E}_2$ are ungraded, 
$$D_1\hat{\boxtimes}D_2:=\begin{pmatrix}0& iD_1\boxtimes\id_{\mathpzc{E}_2}+\id_{\mathpzc{E}_1}\boxtimes D_2\\
-iD_1\boxtimes\id_{\mathpzc{E}_2}+\id_{\mathpzc{E}_1}\boxtimes D_2&0 \end{pmatrix}.$$
The form of these product operators are discussed in for instance \cite[Section 1]{BMS}. We follow the conventions of \cite{hilsumcmodbun,hilsumfol,hilsumbordism}.

\begin{define}
We say that two symmetric chains $(\mathpzc{E}_1,D_1)$ and $(\mathpzc{E}_2,D_2)$ are \emph{compatible for the exterior product} if the operator $D_1\hat{\boxtimes}D_2$ is regular with $(D_1^*\hat{\boxtimes}D_2^*)^*=D_1\hat{\boxtimes}D_2$.
\end{define}

\begin{remark}
It is unclear if there is an example of a pair of symmetric chains which is not compatible for products. It follows from Lemma \ref{sumregul} that a cycle is always compatible for the exterior product with a symmetric chain.
\end{remark}

\begin{prop}
\label{extprodprop}
The exterior product of chains
$$(\mathpzc{E}_1,D_1)\hat{\boxtimes}(\mathpzc{E}_2,D_2):=(\mathpzc{E}_1\hat{\boxtimes}\mathpzc{E}_2, D_1\hat{\boxtimes}D_2),$$
extends to a partially defined biadditive mapping on pairs of compatible symmetric chains
$$C^{\s}_*(\mathcal{A}_1,B_1)\times C^{\s}_*(\mathcal{A}_2,B_2)\dashrightarrow C^{\s}_*(\mathcal{A}_1\otimes^{\alg} \mathcal{A}_2, B_1\minten B_2).$$
\end{prop}

\begin{remark}
\label{bjhilremark}
In  \cite{baajjulg}, Baaj and Julg showed that the exterior product of cycles is again a cycle and that under bounded transform, the exterior Kasparov product in $KK$ defines the same class as the bounded transform of the exterior product in the sense of above. It was shown in \cite[Lemma 3.3]{hilsumbordism} that the same holds true also for half-closed chains; half-closed chains are closed under exterior products and on the level of classes this exterior product is compatible with Kasparov's exterior product.
\end{remark}

\begin{ex}{\bf Exterior products with the real line} \label{psiinf} \\
Consider the $(C^\infty_c(\R),\C)$-cycle given by $(L^2(\R),\partial_x)$, where the differential operator $\partial_x:=i\frac{\mathrm{d}}{\mathrm{d} x}$ is equipped with the domain given by the graph closure of $C^\infty_c(\R)$. Since $\R$ is a complete manifold, $\partial_x$ is a self-adjoint operator with locally compact resolvent. We define the additive mapping 
$$\Psi_\infty:C^\s_*(\mathcal{A},B)\to C^\s_{*+1}(C^\infty_c(\R,\mathcal{A}),B), \quad \Psi_\infty(\mathpzc{E},D):=(\mathpzc{E},D)\hat{\boxtimes}(L^2(\R),\partial_x).$$
In Proposition \ref{extprodprop}, the exterior product $(\mathpzc{E},D)\hat{\boxtimes}(L^2(\R),\partial_x)$ is a chain for the algebraic tensor product $C^\infty_c(\R)\otimes^{\alg} \mathcal{A}$ but it is easily verified to extend to $C^\infty_c(\R,\mathcal{A})$. By Remark \ref{bjhilremark}, this mapping restricts to additive mappings
$$\Psi_\infty:C^{\hc}_*(\mathcal{A},B)\to C^{\hc}_{*+1}(C^\infty_c(\R,\mathcal{A}),B),\quad \Psi_\infty:Z_*(\mathcal{A},B)\to Z_{*+1}(C^\infty_c(\R,\mathcal{A}),B).$$
We often use the notation 
\begin{align}
\nonumber
\Psi_\infty(D)&:=D\hat{\boxtimes}\partial_x\\
&=\begin{cases} i\left(\gamma_\mathpzc{E}\boxtimes\frac{\partial}{\partial t} + \gamma_\mathpzc{E}D\boxtimes\mathrm{id}_{L^2(\R)}\right)&*=0\\
\label{writeingoutpsiinfty}
\begin{pmatrix} 0& \mathrm{id}_\mathpzc{E}\boxtimes\frac{\partial}{\partial t}+ D\boxtimes\mathrm{id}_{L^2(\R)}, \\
- \mathrm{id}_\mathpzc{E}\boxtimes\frac{\partial}{\partial t} + D\boxtimes\mathrm{id}_{L^2(\R)}&0 \end{pmatrix}, &*=1
\end{cases}
\end{align}
\end{ex}

\begin{ex}{\bf Exterior products with the interval} \label{psiint} \\
Consider the $(C^\infty_c(0,1),\C)$-chain given by $(L^2(0,1),\dmin_{x})$, where $\dmin_x$ is the differential expression $i\frac{\mathrm{d}}{\mathrm{d} x}$ equipped with the domain $H^1_0[0,1]$ which coincides with the graph closure of $C^\infty_c(0,1)$. It is easily verified that $\dmin_x$ is symmetric, with $(\dmin_x)^*=\dmax_{x}$ -- the maximal extension of the differential expression $i\frac{\mathrm{d}}{\mathrm{d} x}$ whose domain is $H^1[0,1]$. A short computation with Fourier series shows that $(L^2(0,1),\dmin_{x})$ is a half-closed $(C^\infty_c(0,1),\C)$-chain. In fact, $(L^2(0,1),\dmin_x)$ is compatible for exterior products with any symmetric chain, see below in Theorem \ref{extprodint}. As such, we define an additive mapping 
$$\Psi:C^\s_*(\mathcal{A},B)\to C^\s_{*+1}(C^\infty_c((0,1),\mathcal{A}),B), \quad \Psi(\mathpzc{E},D):=(\mathpzc{E},D)\hat{\boxtimes}(L^2(0,1),\dmin_x).$$
Again, the chains are easily verified to extend to $C^\infty_c((0,1),\mathcal{A})$ and Remark \ref{bjhilremark} shows that the mapping $\Psi$ restricts to an additive mapping on half-closed chains
$$\Psi:C^\hc_*(\mathcal{A},B)\to C^\hc_{*+1}(C^\infty_c((0,1),\mathcal{A}),B).$$
However, the map $\Psi$ does not map cycles to cycles. Similar to the case of $\Psi_\infty$, we let $\Psi(D):=D\hat{\boxtimes}\dmin_x$; it can be expressed as in \eqref{writeingoutpsiinfty}. We will now give the proof of that $(L^2(0,1),\dmin_x)$ is compatible for exterior products with any symmetric chain.
\end{ex}

\begin{theorem}
\label{extprodint}
If $D$ is a symmetric regular operator on a Hilbert $C^*$-module $\mathpzc{E}$, $D\hat{\boxtimes}\dmin_x$ is a symmetric and regular operator on $\mathpzc{E}\boxtimes L^2[0,1]$ whose adjoint is the closure of $D^*\hat{\boxtimes}\dmax_x$.
\end{theorem}

\begin{proof}
The proof is similar to the proof of Theorem \ref{sumregul} apart from one key step; we must prove that $(D^*\hat{\boxtimes}\dmax_x)^*=D\hat{\boxtimes}\dmin_x$ (cf. the statement of \cite[Proposition 7.7]{leschkaad2}). We consider only the even case, the odd follows in a similar way. To follow the notations of \cite{leschkaad2}, we set
 \begin{align*}
 \Dom(T)&=W(D)\boxtimes L^2[0,1],\qquad T:=i\gamma_\mathpzc{E}D\boxtimes \mathrm{id},\\
\Dom(T^*)&=W(D^*)\boxtimes L^2[0,1],\quad T^*:=i\gamma_\mathpzc{E}D^*\boxtimes \mathrm{id}\\
\Dom(S)&= \mathpzc{E}\boxtimes H^1_0[0,1],\qquad\!\!\qquad S:=\gamma_\mathpzc{E}\boxtimes \dmin_x\\
\Dom(S^*)&= \mathpzc{E}\boxtimes H^1[0,1],\quad\,\qquad S^*:=\gamma_\mathpzc{E}\boxtimes \dmax_x .
\end{align*} 
The operators $S$ and $T$ anticommute on $\Dom(ST)=\Dom(TS)$. A direct computation using that $T$ and $S$ are anticommuting symmetric operators with the common core $\Dom(ST)=\Dom(TS)$ shows that for $\xi\in \Dom(S)\cap\Dom(T)$
\begin{equation}
\label{anitcommst}
\langle (S+T)\xi,(S+T)\xi\rangle = \langle S\xi,S\xi\rangle+\langle T\xi,T\xi\rangle.
\end{equation}
Equation \eqref{anitcommst} shows that the operator $S+T$ is closed on $\Dom(S)\cap\Dom(T)$. Moreover, $(S^*+T^*)^*\supseteq S+T$ holds trivially so it remains to prove $(S^*+T^*)^*= S+T$.

Let $\partial_x^{\textnormal{per}}$ be the differential expression $\partial_x$ equipped with periodic boundary conditions on $[0,1]$. This being the Dirac operator on the closed manifold $\R/\Z$, it is a self-adjoint operator on the Hilbert space $L^2[0,1]$. We use the notation $S_{\textnormal{per}}:=\gamma_\mathpzc{E}\boxtimes \partial_x^{\textnormal{per}}$. This is a self-adjoint regular operator on $\mathpzc{E}\hat{\boxtimes} L^2[0,1]$ by \cite[Chapter 10]{lancesbook}. We have that $S^*+T^*$ is an extension of $S_{\textnormal{per}}+T^*$, hence $(S_{\textnormal{per}}+T^*)^*$ is an extension of $(S^*+T^*)^*$. The operators $S_{\textnormal{per}}$ and $T$ trivially satisfy the assumptions of Theorem \ref{sumregul}, hence $(S_{\textnormal{per}}+T^*)^*=S_{\textnormal{per}}+T$. It follows that 
\begin{align*}
\Dom(S^*+T^*)^*&\subseteq\Dom(S_{per}+T)\\
&= \{f\in \mathpzc{E}\hat{\boxtimes}H^1[0,1]\cap \Dom(T)\hat{\boxtimes}L^2[0,1]: \; f(0)=f(1)\}.
\end{align*}
Here we are identifying elements of $\mathpzc{E}\hat{\boxtimes}L^2[0,1]$ with $\mathpzc{E}$-valued functions on $[0,1]$. Take an $f\in \Dom(S^*+T^*)^*$ and a 
$$g\in \Dom(S^*+T^*)=\mathpzc{E}\hat{\boxtimes}H^1[0,1]\cap \Dom(T^*)\hat{\boxtimes}L^2[0,1].$$ 
By partial integration, 
$$\langle f, (S^*+T^*)g\rangle_{\mathpzc{E}\hat{\boxtimes}L^2[0,1]}-\langle (S_{\textnormal{per}}+T)f,g\rangle_{\mathpzc{E}\hat{\boxtimes}L^2[0,1]}=i\langle\gamma f(0),g(1)-g(0)\rangle_\mathpzc{E}.$$
Since this holds for all $g\in \mathpzc{E}\hat{\boxtimes}H^1[0,1]\cap \Dom(T^*)\hat{\boxtimes}L^2[0,1]$, e.g. for $g(x)=xg_0$ and $g(x)=(1-x)g_0$ for all $g_0$ in the dense submodule $\Dom(T^*)\subseteq \mathpzc{E}$ we conclude that any $f\in\Dom(S^*+T^*)^*$ satisfies $f(0)=f(1)=0$. In particular, 
$$\Dom(S^*+T^*)^*\subseteq\mathpzc{E}\hat{\boxtimes}H^1_0[0,1]\cap \Dom(T)\hat{\boxtimes}L^2[0,1]=\Dom(S+T).$$
Hence $(S^*+T^*)^*=S+T$.
\end{proof}

\subsection{Twisting by a projective module}
\label{subseprodkth}

In this section the product between chains and projective modules is considered. This is one of the simplest non-trivial examples available for an unbounded Kasparov product, and is of importance when constructing examples on manifolds. The case of spectral triples was first considered by Connes in \cite{connesgravity} and is described in detail in \cite[Section 2]{chakramathai}. See also \cite{LRV}.

Let $\mathcal{A}$ be a unital dense $*$-subalgebra of a unital $C^*$-algebra $A$. We denote by $\Omega^1(\mathcal{A})$ the space of abstract $1$-forms on $\mathcal{A}$; that is, 
\[\Omega^{1}(\A):=\ker (m: \A\otimes \A\to \A), \quad\mbox{where}\quad m(a_{1}\otimes a_{2}):=a_{1}a_{2}.\]
The map $\der:\A\to \Omega^{1}(\A)$ defined by $\der a:= 1\otimes a- a\otimes 1$ is an $\A$-bimodule derivation. We call a pair $(\mathcal{E}_{\mathcal{A}},\nabla)$ an $\mathcal{A}$-module with connection if $\E_{\A}$ is a finitely generated projective $\mathcal{A}$-module and $\nabla$ is a connection on $\mathcal{E}_{\mathcal{A}}$, i.e. $\nabla:\mathcal{E}_{\mathcal{A}}\to \,\mathcal{E}\otimes_{\mathcal{A}} \Omega^1(\mathcal{A})$ is a linear mapping satisfying the Leibniz rule
$$ \nabla(ea)=\nabla(e)a+e\otimes \der a.$$
Suppose that we are given an $(\mathcal{A},B)$-chain  $(\mathpzc{F},Q)$ and an $\mathcal{A}$-module with connection $(\mathcal{E}_{\mathcal{A}},\nabla)$. Then there is a linear mapping $\nabla_Q:\mathcal{E}_{\mathcal{A}}\to \mathcal{E}\otimes_{\mathcal{A}} \End^*_B(\mathpzc{F})$ obtained by composing $\nabla$ with $\delta_Q:\Omega^1(\mathcal{A})\to \End^*_B(\mathpzc{F})$, $a\mathrm{d} b\mapsto a[Q,b]$. 

We tacitly assume that $\mathcal{E}_{\mathcal{A}}$ has an $\mathcal{A}$-valued inner product $\langle\cdot ,\cdot\rangle_{\mathcal{E}}$. From the inner product $\langle\cdot ,\cdot\rangle_{\mathcal{E}}$, we construct pairings $\E\times \E\otimes_\A\Omega^{1}(\A)\to \Omega^{1}(\A)$ and $\E\otimes_{\A}\Omega^{1}(\A)\times \E\to \Omega^{1}(\A)$, also denoted by $\langle\cdot ,\cdot\rangle_{\mathcal{E}}$. A connection $\nabla$ is said to be \emph{Hermitian} if 
$$\der \langle \xi,\eta\rangle =\langle \xi,\nabla\eta\rangle-\langle\nabla \xi,\eta\rangle.$$
If $\nabla$ is Hermitian, then for any symmetric $(\mathcal{A},B)$-chain $(\mathpzc{F},Q)$ and $\xi,\eta\in \,\!\mathcal{E}_{\mathcal{A}}$,
$$\delta_{Q} \langle \xi,\eta\rangle =\langle \xi,\nabla_{Q}\eta\rangle+\langle\eta, \nabla_{Q} \xi\rangle^{*}.$$
When $\mathcal{E}=p\cdot\mathcal{A}^n$ for a projection $p\in M_n(\mathcal{A})$, the Gra\ss mannian connection $\nabla(p\xi):=p\mathrm{d}(p\xi)$ is always Hermitian.

\begin{define}
\label{ktheprofdefinition}
Let $(\mathpzc{F},Q)$ be an $(\mathcal{A},B)$-chain and $(\mathcal{E}_{\mathcal{A}},\nabla)$ an $\mathcal{A}$-module with connection. We define 
$$(\mathcal{E}_{\mathcal{A}},\nabla)\otimes_{\mathcal{A}} (\mathpzc{F},Q):=(\,\mathcal{E}\otimes_{\mathcal{A}}\mathpzc{F}, 1\otimes_\nabla Q),$$
where $\Dom(1\otimes_\nabla Q):=\,\mathcal{E}\otimes_{\mathcal{A}}\Dom(Q)$ and thereon $1\otimes_\nabla Q$ is defined by
$$1\otimes_\nabla Q(e\otimes f):=e\otimes Qf+\nabla_Q(e)f.$$
\end{define}

\begin{prop}
If $(\mathpzc{F},Q)$ is an $(\mathcal{A},B)$-chain and $(\mathcal{E}_{\mathcal{A}},\nabla)$ and $(\mathcal{E}_{\mathcal{A}},\nabla')$ are $\mathcal{A}$-modules with connection then $\Dom(1\otimes_\nabla Q)=\Dom(1\otimes_{\nabla'} Q)$ and $1\otimes_\nabla Q-1\otimes_{\nabla'} Q$ extends to a bounded adjointable mapping on $\,\mathcal{E}\otimes_{\mathcal{A}}\mathpzc{F}$.
\end{prop}

\begin{proof} 
This follows from the fact that the space of connections is an affine space modeled on $\Hom_{\mathcal{A}}(\mathcal{E},\mathcal{E}\otimes_{\mathcal{A}}  \Omega^1(\mathcal{A}))$.
\end{proof}

For a projective $\A$-module, we denote by $\End^{*}_{\A}(\E)$ the algebra of $\A$-linear endomorphisms $T:\E\to \E$. Such endomorphisms always admit an adjoint for the inner product, since $\A$ is a $*$-algebra.

\begin{lemma}
\label{prodwithkthe}
If $(\mathpzc{F},Q)$ is a symmetric $(\mathcal{A},B)$-chain and $(\mathcal{E}_{\mathcal{A}},\nabla)$ is an $\mathcal{A}$-module with a Hermitian connection then $(\,\mathcal{E}_{\mathcal{A}},\nabla)\otimes_{\mathcal{A}} (\mathpzc{F},Q)$ is a well defined symmetric $(\End^{*}_{\A}(\E), B)$-chain which is half-closed respectively closed if $(\mathpzc{F},Q)$ is half-closed respectively closed. 
\end{lemma}

\begin{proof}
If $(\mathpzc{F},Q)$ is closed, the result follows from \cite[Theorem 6.2.7]{thebeastofmesland}. If $(\mathpzc{F},Q)$ is symmetric, choose a projection $p\in M_{n}(\A)$ defining $\E\cong p \A^{n}$, and consider the module $\A^{n}\otimes_{\A}\mathpzc{F}\cong \mathpzc{F}^{n}$ and the mutually adjoint regular diagonal operators $Q$ and $Q^{*}$. By Remark \ref{LipDstar} the projection $p$ is an element of both $\Lip(Q)$ and $\Lip(Q^{*})$. Thus the result follows from the proof of \cite[Theorem 6.2.7]{thebeastofmesland} after considering the symmetric chain $(\mathpzc{F}\oplus \mathpzc{F},\tilde{Q})$ with the self-adjoint regular operator $\tilde{Q}$, for notation see Equation \eqref{twiddlingoperators} on page \pageref{twiddlingoperators}.
\end{proof}

Let $A'$ be another choice of $C^{*}$-algebra with a choice of dense $*$-subalgebra $\A'$. By a \emph{Hermitian} $(\A',\A)$-\emph{module with connection} we mean a pair $(_{\A'}\E_{\A},\nabla)$ consisting of an $(\A',\A)$-bimodule $_{\A'}\E_{\A}$ with Hermitian connection $\nabla:\E_{\A}\to \E\otimes_{\A}\Omega^{1}(A)$ on the right module $\E_{\A}$. 
Two Hermitian $(\mathcal{A}',\mathcal{A})$-modules with connection are said to be isomorphic if there is a unitary isomorphism of $(\mathcal{A}',\mathcal{A})$-bimodules compatible with the choice of connection. We define $\Her(\mathcal{A}',\mathcal{A})$ as the semigroup of isomorphism classes of $(\mathcal{A}',\mathcal{A})$-modules with a Hermitian connection. 

We remark that the forgetful mapping from the semigroup  $\Her(\mathcal{A}',\mathcal{A})$ to that of finitely generated right projective $(\mathcal{A}',\mathcal{A})$-bimodules is surjective, i.e. any such module admits the Gra\ss man connection. The proof of the next proposition follows immediately from Lemma \ref{prodwithkthe}. 

\begin{prop}
The product of Definition \ref{ktheprofdefinition} defines a biadditive map
$$\Her(\mathcal{A}',\mathcal{A})\times C^\s_*(\mathcal{A},B)\to C^\s_*(\mathcal{A}',B).$$
Furthermore, it maps the sub-semigroups of half-closed, respectively closed $(\mathcal{A},B)$-chains to half-closed, respectively closed $(\mathcal{A}',B)$-chains. 
\end{prop}

\begin{remark}
A fact that will be of use later for manifolds is that if $\mathcal{A}'\subseteq \mathcal{A}$ is a central subalgebra, any right $\mathcal{A}$-module is in the obvious way an $(\mathcal{A}',\mathcal{A})$-bimodule giving rise to an additive mapping $\Her(\C,\mathcal{A})\to \Her(\mathcal{A}',\mathcal{A})$. 
\end{remark}

\subsection{Dirac operators on manifolds}
\label{mfdwbdryex}

In this subsection, we discuss how examples of chains as in Definition \ref{cycleschaindeef} naturally appear in manifold theory. We fix a unital $C^*$-algebra $B$. We will refer to a smooth locally trivial bundle of finitely generated $B$-Hilbert $C^*$-modules as a Hermitian $B$-bundle. We start with the following structures:
\begin{enumerate} 
\item $W$ is a Riemannian spin$^c$-manifold with boundary;
\item $E_B\to W$ is a Hermitian $B$-bundle;
\item $\nabla_E$ denotes a $B$-linear Hermitian connection on $E_B$.
\end{enumerate}

For now, we allow $W$ to be noncompact; we will mainly focus on the case of $W$ is complete or a compact manifold with boundary. Further details on this setup can be found in \cite{deelgoffIII,hankpapsch,Schickconn}. We tacitly assume that all structures (the metric on $W$, $E_B$ and its connection $\nabla_E$) are of product type near $\partial W$. Let $S_W\to W$ denote the spinor bundle on $W$. If $\dim(W)$ is even, $S_W$ can be considered a graded bundle. If $N$ is large enough, there is a projection $p_{\C\ell}\in M_N(C^\infty(W,B))$ such that 
$$C(W,S_W\otimes E_B)\cong p_{\C\ell}C(W,B^N).$$ 
Consider the $B$-Hilbert $C^*$-module,
$$\mathpzc{E}_W:=L^2(W,S_W\otimes E_B)\cong p_{\C\ell}L^2(W,B^N).$$
Here we define $L^2(W,B^N):=L^2(W)\otimes B^N$ -- the $B$-Hilbert $C^*$-module completion of the algebraic tensor product $L^2(W)\otimes^{\mathrm{alg}} B^N$ where $B^N$ is equipped with its standard $B$-Hilbert $C^*$-module structure. If $\dim(W)$ is even, $\mathpzc{E}_W$ forms a graded $B$-Hilbert $C^*$-module in the grading induced from $S_W$. We define the Banach $*$-algebras 
\begin{align*}
\Lip_0(W)&:=\{f\in C_0(W): \;\mathrm{d}f\in L^\infty(W,T^*W)\},\quad\mbox{and}\\
&\Lip_0(W^\circ):=\Lip_0(W)\cap C_0(W^\circ).
\end{align*}
We can identify 
$$H^1(W, B^N)=\{f\in L^2(W,B^N):\; \mathrm{d}f\in L^2(W,T^*W\otimes B^N)\}.$$
Here $H^1(W, B^N):=H^1(W)\otimes B^N$ denotes the $B$-Hilbert $C^*$-module completion of the algebraic tensor product $H^1(W)\otimes^{\mathrm{alg}} B^N$. We can extend the trace mapping $H^1(W)\otimes^{\mathrm{alg}} B^N\to  L^2(\partial W)\otimes^{\mathrm{alg}} B^N$ by continuity to an adjointable $B$-linear mapping $H^1(W, B^N)\to L^2(\partial W, B^N)$, $f\mapsto f|_{\partial W}$. Also, let
\begin{align*}
\mathcal{E}_W^{\mini}:&=H^1_0(W,S_W\otimes E_B)\\
&=\{f\in p_{\C\ell}L^2(W,B^N):\; \mathrm{d}f\in L^2(W,T^*W\otimes B^N),\; f|_{\partial W}=0\}.
\end{align*}
The module $\mathcal{E}_W^{\mini}$ is a dense submodule of $\mathpzc{E}_W$ with a left $\Lip_0(W)$-action. We construct the spin$^c$-Dirac operator $D^W$ on $W$ from the Riemannian metric: the operator $D^W_E$ is the differential expression acting on $C^\infty(W^\circ,S_W\otimes E_B)$ as $D^W$ twisted by the connection $\nabla_E$ equipped with the domain $\mathcal{E}_W^{\textnormal{min}}$. 

\begin{lemma}
\label{firstgammadef}
Let $W$ denote a complete manifold or a compact manifold with boundary. Then
$$\gamma_0(W,E_B,\nabla_E):=(\mathpzc{E}_W,D^W_E),$$
defines a symmetric $(\Lip_0(W),B)$-chain which restricts to a half-closed $(\Lip_0(W^\circ),B)$-chain. If $W$ is complete, $\gamma_0(W,E_B,\nabla_E)$ is a $(\Lip_0(W),B)$-cycle.
\end{lemma}

\begin{proof}
If the manifold $W$ is complete, it follows from \cite[Theorem 2.3]{hankpapsch} that the pair  $(\mathpzc{E}_W,D^W_E)$ is a $(\Lip_0(W),B)$-cycle. If $W$ is a compact manifold with boundary, we use the notation $(\Lip(W,E_B),\nabla_E)$ for the $(\Lip(W),\Lip(W,B))$-module with connection defined from $(E_B,\nabla_E)$. On $W$ the spin$^c$-Dirac operator with its minimal domain $H^1_0(W,S_W)$ will be denoted by $D_W^{\textnormal{min}}$, it is clear that $(L^2(W,S_W),D_W^{\textnormal{min}})$ is a symmetric $(\Lip(W),\C)$-chain restricting to a half-closed $(\Lip_0(W^\circ),\C)$-chain. We let $(B,0)$ denote the obvious $(\C,B)$-cycle. The result follows from noting that 
$$ \gamma_0(W,E_B,\nabla_E)=(\Lip(W,E_B),\nabla_E)\otimes_{\Lip(W,B)} \left((L^2(W,S_W),D_W^{\textnormal{min}})\hat{\boxtimes} (B,0)\right).$$
and applying the results of Subsections \ref{extprodofsymsect} and \ref{subseprodkth}.
\end{proof}

\begin{remark}
It follows from the proof of Lemma \ref{firstgammadef} that $(D^W_E)^*$ is the operator defined from equipping the differential expression given by $1_{E_B}\otimes_{\nabla_E} D^W$ with its maximal domain $\mathcal{E}_W^{\textnormal{max}}$ which is the graph closure of $H^1(W,S_W\otimes E_B)$.
\end{remark}

Let $X$ be a Riemannian manifold and $Y\subseteq X$ a smooth sub-manifold. The subalgebra $\Lip_0(X\setminus Y)\subseteq \Lip_0(X)$ is a closed ideal. A function $f:W\to X$ that is a globally Lipschitz map and satisfies $f(\partial W)\subseteq Y$, will be denoted by  
$$f:(W,\partial W)\to (X,Y).$$
The map $f$ induces a $*$-homomorphism $f^*:\Lip_0(X)\to \Lip_0(W)$ which restricts to a $*$-homomorphism $f^*:\Lip_0(X\setminus Y)\to \Lip_0(W^\circ)$. The next proposition follows from Proposition \ref{semiisoclass} and Lemma \ref{firstgammadef}.

\begin{prop}
With $(W,E_B,\nabla_E)$ as above and $f:(W,\partial W)\to (X,Y)$ being a globally Lipschitz mapping, $f_*(\mathpzc{E}_W,D^W_E):=(f^*)^*(\mathpzc{E}_W,D^W_E)$ is a symmetric $(\Lip_0(X),B)$-chain restricting to a half-closed $(\Lip_0(X\setminus Y),B)$-chain. If $W$ is complete, $(\mathpzc{E}_W,D^W_E)$ is a  $(\Lip_0(X),B)$-cycle.
\end{prop}

\section{Bordisms of unbounded $KK$-cycles}
\label{bordismsection}

The purpose of this section is to present a notion of bordism in unbounded $KK$-theory stemming from \cite{hilsumcmodbun, hilsumfol, hilsumbordism}. The goal is to show that bordism defines an equivalence relation on the semigroup of unbounded $KK$-cycles. Beyond establishing a toolbox for working with bordisms in unbounded $KK$-theory, the main result of this section states that the bordism group of closed cycles indeed is an abelian group, which maps to the standard $KK$-group via the bounded transform.

\subsection{Symmetric chains with boundary}

Before defining what a bordism is, we will need a notion of boundary of symmetric chains. The definition is due to Hilsum; it is found in \cite[Section 3]{hilsumcmodbun} in the odd case and in \cite[Section 5]{hilsumbordism} in the even case. In fact, a number of the results in this subsection are due to Hilsum; we collect them here for convience.

The notion of the boundary of a symmetric chain used here is similar to the condition on a manifold of requiring the existence of a collar neighborhood where all the geometric structures are of product type. We remark that for the topological applications we have in mind, such a ``classical" notion suffices albeit a more ``noncommutative" notion of a boundary remains an interesting and tractable goal for future research. Recall the notion of dimension modulo $2$ from Definition \ref{cycleschaindeefnew}.

\begin{define}
Let $(\mathpzc{F}, Q)$ be a symmetric $(\mathcal{A},B)$-chain of dimension $k\mod 2$ and $(\mathpzc{E}, D)$ be an $(\mathcal{A},B)$-cycle of dimension $k-1\mod 2$. We say that $(\theta,p)$ is \emph{boundary data for} $(\mathpzc{E}, D)$ \emph{relative to} $(\mathpzc{F}, Q)$ if 
\begin{enumerate}
\item $p\in \End^*_B(\mathpzc{F})$ is a projection commuting with the $A$-action, required to be even if $k=0\mod 2$.
\item $\theta:p\mathpzc{F}\xrightarrow{\sim} L^2[0,1]\hat{\boxtimes} \mathpzc{E}$ is an isomorphism of $(A,B)$-Hilbert $C^*$-modules, required to be graded if $k=0\mod 2$.
\end{enumerate}
\end{define}

We equip $L^2[0,1]$ with the point wise action of $C[0,1]$. The following construction is clear from the definition and should be compared with \cite[Examples 3.3 and 3.7]{hilsumcmodbun}.

\begin{prop}
Assume that $(\theta,p)$ is boundary data for the symmetric chain $(\mathpzc{F}, Q)$ relative to the closed cycle $(\mathpzc{E}, D)$. For $\phi\in C[0,1]$ and $a\in A$, we set 
\begin{equation}\label{b}b(\phi\otimes a):=\left(\theta(\phi\otimes \mathrm{id}_{\mathpzc{E}})\theta^{-1}\cdot p +\phi(1)(1-p)\right)\pi(a).\end{equation}
Under $b$, $\mathpzc{F}$ forms a $(C[0,1]\otimes A,B)$-Hilbert $C^*$-module.
\end{prop}

We sometimes shorten notation and write $b(\phi)$ for $b(\phi\otimes 1)$. Recall the construction of $\Psi(D)$ from Example \ref{psiint} (see page \pageref{psiint}).

\begin{define}
\label{boundarydef}
Let $(\mathpzc{F}, Q)$ be a symmetric chain and $(\mathpzc{E}, D)$ a closed $(\mathcal{A},B)$-cycle. We say that $(\mathpzc{E}, D)$ is a \emph{boundary of} $(\mathpzc{F}, Q)$ \emph{with boundary data} $(\theta,p)$ if 
\begin{enumerate}
\item For $\phi\in C^\infty_c(0,1]$,  
$$b(\phi)\Dom Q^*\subseteq \Dom Q\quad\mbox{and}\quad Q^*b(\phi)=Qb(\phi)\quad\mbox{on}\quad \Dom Q^*.$$
\item For $\phi\in C^\infty_c(0,1)$,  
$$\phi \Dom \Psi(D)=\theta b(\phi)\Dom Q \quad \mbox{and} \quad Q=\theta^{-1}\Psi(D)\theta\quad \mbox{on}\quad b(\phi)\Dom Q.$$
\item For $\phi_1,\phi_2\in C^\infty[0,1]$ satisfying $\phi_1\phi_2=0$,  
$$b(\phi_1) Q b(\phi_2)=0.$$
\end{enumerate}
\end{define}

\begin{remark}
We remark here that the specific choice of the interval $[0,1]$ is not important. We will in fact start to work with general compact intervals $[a,b]$ when constructing boundaries. Below, in Subsection \ref{homotopylemmasection}, we will see that up to the soon-to-be defined notion of bordism the choice of interval is irrelevant. 

\end{remark}

\begin{lemma}
\label{cylinderconstruction}\cite[Example 5.3]{hilsumbordism}. Any closed $(\mathcal{A},B)$-cycle is the boundary of a symmetric $(\mathcal{A},B)$-chain.
\end{lemma}

\begin{proof}
Let $(\mathpzc{E}, D)$ be a closed $(\mathcal{A},B)$-cycle. Using Subsection \ref{extprodofsymsect}, and Theorem \ref{sumregul} in particular, we can define the symmetric chain
$$(\mathpzc{E}_{\cyl},D_{\cyl}):=(\mathpzc{E},D)\hat{\boxtimes} (L^2[0,\infty),\partial_x^{\textnormal{min}}),$$
where $\partial_x^{\textnormal{min}}:L^2[0,\infty)\dashrightarrow L^2[0,\infty)$ is the minimal extension of the differential expression $i\frac{\mathrm{d}}{\mathrm{d}x}$. It is clear that the projection $p_{\cyl}$, defined by multiplication by the characteristic function of $[0,1]$, and $\theta_{\cyl}=\mathrm{id}_{L^2[0,1]\otimes \mathpzc{E}}$ satisfies the assumptions of Definition \ref{boundarydef}. Hence $(\mathpzc{E}, D)$ is the boundary of $(\mathpzc{E}_{\cyl},D_{\cyl})$ with boundary data $(\theta_{\cyl},p_{\cyl})$.
\end{proof}

\subsection{Bordisms of closed cycles}

For a locally convex topological vector space $\A$, we use the notation
$$C^1_0((0,1],\mathcal{A}):=\left\{\varphi\in C_0((0,1],\mathcal{A})\cap C^1([0,1],\mathcal{A}): \;\varphi'(0)=0\right\}.$$
Equivalently, if $\mathcal{A}$ is a Banach space, $C^1_0((0,1],\mathcal{A})$ is the closure of $C^\infty_c((0,1],\mathcal{A})$ inside $\Lip([0,1],\mathcal{A})$. Similarly, we define $C^1_0((0,1),\mathcal{A})$ which coincides with the closure of $C^\infty_c((0,1),\mathcal{A})$ inside $\Lip([0,1],\mathcal{A})$ if $\mathcal{A}$ is a Banach space.

\begin{prop}[Lemma 5.4 of \cite{hilsumbordism}]
\label{cyclessss}
Assume that $(\mathpzc{E}, D)$ is a boundary of $(\mathpzc{F}, Q)$ with boundary data $(\theta,p)$. The $(C_0((0,1],A),B)$-Hilbert $C^*$-module structure on $\mathpzc{F}$ induced by $b$ in Equation \eqref{b} makes $(\mathpzc{F}, Q)$ into a symmetric $(C^1_0((0,1],\mathcal{A}),B)$-chain such that $b(C^1_0((0,1],\mathcal{A}))\subseteq \End_B^*(Q^*,Q)$. This symmetric chain restricts to a half-closed $(C^1_0((0,1),\mathcal{A}),B)$-chain. 
\end{prop}

\begin{define}
If $(\mathpzc{E}, D)$ is a boundary of $(\mathpzc{F}, Q)$ with boundary data $(\theta,p)$ such that the associated symmetric $(C^\infty_c((0,1],\mathcal{A}),B)$-chain $(\mathpzc{F}, Q)$ is half-closed, then we say that $(\mathpzc{F}, Q,\theta,p)$ is a \emph{bordism with boundary} $(\mathpzc{E}, D)$. We write 
$$\partial(\mathpzc{F}, Q,\theta,p)=(\mathpzc{E}, D).$$
If $(\mathpzc{E}, D)$ is the boundary of a bordism, we say that $(\mathpzc{E}, D)$ is \emph{nullbordant}.

Two closed cycles $(\mathpzc{E}_1, D_1)$ and $(\mathpzc{E}_2, D_2)$ are said to be \emph{bordant}, written 
$$(\mathpzc{E}_1, D_1)\sim_{\bor}(\mathpzc{E}_2, D_2),$$ 
if $(\mathpzc{E}_1, D_1)+(-(\mathpzc{E}_2, D_2))$ is nullbordant. 
\end{define}

\begin{remark}
\label{bordismandbddtrans}
By \cite[Theorem 6.2]{hilsumbordism},  if $(\mathpzc{E},D)\sim_{\bor} 0$, then $[\mathpzc{E},b(D)]=0$ in $KK_*(A,B)$. This fact is the main motivation for studying bordism as a relation on unbounded $KK$-cycles.
\end{remark}

\begin{define}
We say that a bordism $(\mathpzc{F},Q,\theta,p)$ has \emph{empty boundary} if $p=0$.
\end{define}

\begin{prop}
A bordism with empty boundary defines a closed cycle.
\end{prop}

\begin{proof}
If $(\mathpzc{F},Q,0,0)$ is a bordism, $b(\phi\otimes a)=\phi(1)\pi(a)$. Take $\phi_0\in C^\infty_c(0,1]$ with $\phi_0(1)=1$, hence $b(\phi_0)=1$. Since $(\mathpzc{F},Q,0,0)$ is a bordism the associated $(C^\infty_c((0,1],\mathcal{A}),B)$-chain is half-closed and
$$\Dom(Q^*)=b(\phi_0)\Dom(Q^*)\subseteq \Dom Q$$ 
so $Q=Q^*$. Furthermore $a(i\pm Q)^{-1}$ are compact for any $a\in A$, because \[a(i\pm Q)^{-1}=b(\phi_0\otimes a)(i\pm Q)^{-1}\] is compact.
\end{proof}

The motivation for the term ``bordism" in this context comes from the fact that often a bordism in the geometric sense induces a bordism in unbounded $KK$-theory; a list of geometric examples which lead to bordisms in the sense of Definition \ref{boundarydef} are given in Subsection \ref{exHilsumSubSec}. The content of the next proposition deals with the classical case of a manifold with boundary; its proof follows directly from Lemma \ref{firstgammadef}.

\begin{prop}
\label{honestbordismprop}
Suppose that $(W,E_B,\nabla_E)$ is data as in Subsection \ref{mfdwbdryex} and $W$ is a compact manifold with boundary. We let $p$ denote the characteristic function of a collar neighborhood $U$ of $\partial W$ and 
$$\theta:L^2(U,(S_W\otimes E_B)|_U)\to L^2[0,1]\boxtimes L^2(\partial W,S_{\partial W}\otimes E_B|_{\partial W}),$$
the geometric isomorphism constructed from the isometry $U\cong [0,1]\times \partial W$. Then, $(\gamma_0(W,E_B,\nabla_E),\theta,p)$ is a bordism with boundary
$$\partial (\gamma_0(W,E_B,\nabla_E),\theta,p)=\gamma_0(\partial W,E_B|_{\partial W},\nabla_E|_{\partial W}).$$
\end{prop}

\subsection{The homotopy lemma}
\label{homotopylemmasection}

One of the main tools for constructing bordisms will be an analogue of operator homotopy for unbounded operators. Recall that the domain of $D$ forms a $B$-Hilbert $C^*$-module $W(D)$ in the $B$-valued inner product, see \eqref{dominprodef} on page \pageref{dominprodef}.

\begin{define}
Let $\mathpzc{W}$ and $\mathpzc{E}$ be $B$-Hilbert $C^*$-modules with left adjointable actions by an algebra $\mathcal{A}$. We say that an $(\mathcal{A},B)$-linear adjointable mapping $T:\mathpzc{W}\to \mathpzc{E}$ is $\mathcal{A}$-locally compact if $aT$ is $B$-compact for all $a\in \mathcal{A}$. If $\mathcal{A}$ is understood from context, we drop $\mathcal{A}$ from the notation.
\end{define}

\begin{prop}
\label{technicalpartofhomtop}
Let $\iota:\mathpzc{W}\to \mathpzc{E}$ be a dense locally compact inclusion of $B$-Hilbert $C^*$-modules. Assume that $D\in C^1([0,1],\Hom^*_B(\mathpzc{W},\mathpzc{E}))$ is a path such that $D(t)$ is a self-adjoint and regular operator partially defined on $\mathpzc{E}$ with $\Dom(D(t))=\iota(\mathpzc{W})$ for any $t$. Consider the operator:
\begin{equation}
\label{internalprdq}
Q:=\begin{cases}
i\gamma_{\mathpzc{E}}\left(\frac{\partial}{\partial t}+D(t)\right), &\mbox{in the even case};\\
\,\\
\begin{pmatrix}
0& \frac{\partial}{\partial t}+D(t)\\
-\frac{\partial}{\partial t}+D(t)& 0
\end{pmatrix}
&\mbox{in the odd case},
\end{cases}
\end{equation}
equipped with the domain $\Dom(Q)=\mathpzc{W}\hat{\boxtimes}L^2[0,1]\cap \mathpzc{E}\hat{\boxtimes}H^1_0[0,1]$. Then $Q$ is a regular symmetric operator whose adjoint $Q^*$ is the graph closure of differential expression \eqref{internalprdq} equipped with the domain $\mathpzc{W}\hat{\boxtimes}L^2[0,1]\cap \mathpzc{E}\hat{\boxtimes}H^1[0,1]$.
\end{prop}

\begin{proof} 
The assumption $D\in C^1([0,1],\Hom^*_B(\mathpzc{W},\mathpzc{E}))$ guarantees relatively bounded anticommutators of $D$ with both $\dmin_t$ and  $\dmax_t$. Therefore the statement follows from Theorem \ref{sumregul}.
\end{proof}

The operator can be viewed as a symmetric analogue of the internal Kasparov product of $D$ with $\dmin_t$. We believe that this viewpoint will be fruitful in future work.

\begin{lemma}
\label{difftopyofoperator}
If $D\in C^1([0,1],\Hom^*_B(\mathpzc{W},\mathpzc{E}))$ satisfies the conditions of Proposition \ref{technicalpartofhomtop} and $(\mathpzc{E},D(t))$ is a cycle for all $t$, then there is a bordism
$$(\mathpzc{E}, D(0))\sim_{\bor} (\mathpzc{E}, D(1)).$$
\end{lemma}

\begin{proof}
By Proposition \ref{technicalpartofhomtop} $(\mathpzc{E}\hat{\boxtimes} L^2[0,1],Q)$ is a symmetric chain. We let $p_1$ denote multiplication by the characteristic function of $[0,1/3]$, $p_2$ likewise for the characteristic function of $[2/3,1]$ and $p=p_1+p_2$. We take $\theta$ to be the identity on 
$$p(\mathpzc{E}\hat{\boxtimes} L^2[0,1])=\mathpzc{E}\hat{\boxtimes} L^2[0,1/3]\oplus \mathpzc{E}\hat{\boxtimes} L^2[2/3,1].$$
We note that the associated action of $C[0,1]\otimes A$ on $\mathpzc{E}\hat{\boxtimes} L^2[0,1]$ is defined by 
$$b(\phi\otimes a)f(t)=
\begin{cases}
\phi(3t)\pi(a)f(t), \quad&t\in [0,1/3],\\
\phi(1)\pi(a)f(t), \quad&t\in (1/3,2/3),\\
\phi(3-3t)\phi(a)f(t), \quad &t\in [2/3,1].
\end{cases}$$

By construction, $(\mathpzc{E},D(0))-(\mathpzc{E},D(1))$ is a boundary for $(\mathpzc{E}\hat{\boxtimes} L^2[0,1],Q)$ with boundary data $(\theta,p)$. By the same argument as that proving Proposition \ref{cyclessss}, see \cite[Lemma 5.4]{hilsumbordism}, it follows that $(\mathpzc{E}\hat{\boxtimes} L^2[0,1],Q)$ is a half-closed $(C^\infty_c((0,1],\mathcal{A}),B)$-chain. Hence $\partial(\mathpzc{E}\hat{\boxtimes} L^2[0,1],Q,\theta,p)=(\mathpzc{E},D(0))-(\mathpzc{E},D(1))$. 
\end{proof}

\begin{cor}
\label{symmetryofbordism}
For any cycle $(\mathpzc{E}, D)$, there is a bordism $(\mathpzc{E}, D)\sim_{\bor}(\mathpzc{E}, D)$. In particular, $(\mathpzc{E}, D)+(-(\mathpzc{E}, D))\sim_{\bor} 0$.
\end{cor}

\begin{proof}
Apply Lemma \ref{difftopyofoperator} to the constant path $D(t)=D$.
\end{proof}

\begin{cor}
If $(\mathpzc{E},D)$ is an odd closed cycle and $s>0$, $(\mathpzc{E},D)\sim_{\bor} (\mathpzc{E},sD)$. If $(\mathpzc{E},D)$ is an even closed cycle, the same holds true for any $s\neq 0$. 
\end{cor}

\begin{proof}
For any $s>0$ we choose $\chi_s\in C^\infty[0,1]$ such that $\chi_s(t)=1$ near $t=0$, $\chi_s(t)=s$ near $t=0$ and $\chi_s'$ having semidefinite sign. For an odd closed cycle $(\mathpzc{E},D)$ and $s>0$ the statement of the Corollary follows by applying Lemma \ref{difftopyofoperator} to the path
$$D(t)=\chi_s(t)D.$$

We let $H$ denote the Heaviside function, i.e. $H(s)=0$ of $s<0$ and $H(s)=1$ if $s>0$. For an even closed cycle $(\mathpzc{E},D)$, with 
$$D=\begin{pmatrix}
0& D_-\\
D_+& 0\end{pmatrix},$$
in the grading. The required result now follows: for any $s\neq 0$, one can apply Lemma \ref{difftopyofoperator} to the path
$$D(t)=\chi_{|s|}(t)\begin{pmatrix}
0& \mathrm{e}^{\pi i tH(-s)}D_-\\
\mathrm{e}^{-\pi i tH(-s)}D_+& 0\end{pmatrix}.$$
\end{proof}

We say that a symmetric operator $T$ is \emph{locally bounded} if $\mathcal{A}$ preserves $\Dom(T)$ and $aT$ as well as $Ta$ have bounded adjointable extensions for any $a\in \mathcal{A}$. If $\Dom(D)\subseteq\Dom(T)$, we say that $T$ has relative bound $s>0$ to $D$ if there is a $t\geq 0$ such that:
$$\|Tx\|\leq s\|Dx\|+t\|x\|,\quad\forall x\in \Dom(D).$$

\begin{cor}
\label{bddpert}
Let $(\mathpzc{E},D)$ be a closed cycle and $T:\mathpzc{E}\dashrightarrow\mathpzc{E}$ a locally bounded symmetric operator (odd if $(\mathpzc{E},D)$ is even) which is relatively bounded to $D$ with relative bound $<1$. There is a bordism
$$(\mathpzc{E},D)\sim_{\bor} (\mathpzc{E},D+T).$$
\end{cor}

\begin{proof}
Consider the path $D(t):=D+tT$ for $t\in [0,1]$. If the relative bound of $T$ is $<1$, the Kato-Rellich theorem for Hilbert modules (see for instance \cite[Theorem 4.5]{leschkaad2}) implies that $D(t)$ is self-adjoint and regular on the domain $\Dom(D)$. Moreover, for any $t\in [0,1]$, since $\Dom(D)=\Dom(D+tT)$ it holds that $a(D\pm i)^{-1}$ is compact if and only if $a(D+tT\pm i)^{-1}$ is compact. Therefore $(\mathpzc{E},D(t))$ is a cycle for any $t$. The proof is complete once applying Lemma \ref{difftopyofoperator} to the path $D(t)$.
\end{proof}

\begin{lemma}
\label{reflexivityofbordism}
If $(\mathpzc{E}_1, D_1)\sim_{\bor}(\mathpzc{E}_2, D_2)$, then $(\mathpzc{E}_2, D_2)\sim_{\bor}(\mathpzc{E}_1, D_1)$. 
\end{lemma}

\begin{proof}
Assume that $\partial(\mathpzc{F},Q,\theta,p)=(\mathpzc{E}_1, D_1)-(\mathpzc{E}_2, D_2)$. If the closed cycles under consideration are even, $\partial(\mathpzc{F},-Q,\theta,p)=(\mathpzc{E}_2, D_2)-(\mathpzc{E}_1, D_1)$. If the closed cycles under consideration are odd, $\partial(-\mathpzc{F},-Q,\theta,p)=(\mathpzc{E}_2, D_2)-(\mathpzc{E}_1, D_1)$. 
\end{proof}

\subsection{The gluing lemma}

Our aim is a proof that bordism forms an equivalence relation on the semigroup of isomorphism classes of closed $(\mathcal{A},B)$-cycles as defined in Proposition \ref{semiisoclass} and that the quotient forms an abelian group. This involves generalizing the gluing construction of Hilsum (see \cite[Section 7]{hilsumcmodbun}). The first step is the following proposition which follows from the definition of direct sum.

\begin{prop}
\label{additivityofcor}
Assume that we are given a direct sum decomposition
$$\partial(\mathpzc{F}, Q,\theta,p)=(\mathpzc{E}_1, D_1)+(\mathpzc{E}_2, D_2)+\cdots +(\mathpzc{E}_k, D_k),$$
for closed $(\mathcal{A},B)$-cycles $(\mathpzc{E}_1, D_1)$, $(\mathpzc{E}_2, D_2)$,...,$(\mathpzc{E}_k, D_k)$. There exists
\begin{enumerate}
\item mutually orthogonal projections $p_1,p_2,\ldots, p_k\in \End^*_B(\mathpzc{F})$ commuting with the $A$-action such that 
$$p=\sum_{j=1}^kp_j;$$
\item isomorphisms $\theta_j:p_j\mathpzc{F}\xrightarrow{\sim} \mathpzc{E}_j\hat{\boxtimes}L^2[0,1]$ of $(A,B)$-Hilbert $C^*$-modules such that $\theta=\oplus_{j=1}^k \theta_j$.
\end{enumerate}
\end{prop}

We use the suggestive notation $(\theta,p)=(\theta_1,p_1)\dot{\cup}(\theta_2,p_2)\dot{\cup}\cdots \dot{\cup}(\theta_k,p_k).$

\begin{theorem}
\label{gluingthm}
Assume that $(\mathpzc{F}, Q)$ and $(\mathpzc{F}', Q')$ are two bordisms with boundary data $(\theta,p)$ respectively $(\theta',p')$ and
\begin{align*}
\partial(\mathpzc{F}, Q,\theta,p)&=(\mathpzc{E}_1,D_1)+(\mathpzc{E}_2,D_2)\quad\mbox{and}\\
\partial(\mathpzc{F}', Q',\theta',p')&=-(\mathpzc{E}_2,D_2)+(\mathpzc{E}_3,D_3),
\end{align*}
for some closed cycles $(\mathpzc{E}_1,D_1)$, $(\mathpzc{E}_2,D_2)$ and $(\mathpzc{E}_3,D_3)$. Then there is a bordism $(\mathpzc{F}'', Q'',\theta'',p'')$ such that 
\begin{enumerate}
\item The $(A,B)$-Hilbert $C^*$-module $\mathpzc{F}''\subseteq \mathpzc{F} \oplus \mathpzc{F}' $ is determined by the pullback diagram
\[
\begin{CD}
\mathpzc{F}'' @>>>\mathpzc{F} \\
@VVV @ VV\theta_2p_2 V\\
\mathpzc{F}' @>\sigma\theta_2'p_2'>>\mathpzc{E}_2\hat{\boxtimes}L^2[0,1]  \\
\end{CD}, \]
where $\sigma: \mathpzc{E}\hat{\boxtimes}L^2[0,1]\to \mathpzc{E}\hat{\boxtimes}L^2[0,1]$ is defined by 
$$\sigma\xi(t)=\xi(1-t),$$
and where we have written 
$$(\theta,p)=(\theta_1,p_1)\dot{\cup}(\theta_2,p_2)\quad\mbox{and}\quad (\theta',p')=(\theta_2',p_2')\dot{\cup}(\theta_3',p_3').$$ 
\item The operator $Q''$ has domain 
$$\quad\quad\Dom Q''=\left\{z=(x,x'):\; 
\begin{cases}
x\in (1-p_2+\theta_2^{-1}(\chi_1\otimes 1)\theta_2p_2)\Dom Q,\\
x'\in (1-p_2'+(\theta_2')^{-1}(\chi_2\otimes 1)\theta_2'p_2')\Dom Q'.
\end{cases}\right\},$$
where $\chi_1,\chi_2\in C^\infty_c(0,1]$ satisfy $\chi_1(t)+\chi_2(1-t)=1,$ and $Q''$ acts by 
$$Q''z=(Qx,Q'x').$$
\item The boundary data $(\theta'',p'')$ is defined as $(\theta_1,p_1)\dot{\cup}(\theta_3',p_3')$.
\end{enumerate}
The bordism $(\mathpzc{F}'', Q'',\theta'',p'')$ does not depend on the choice of $\chi_1$ and $\chi_2$ and satisfies
$$\partial(\mathpzc{F}'', Q'',\theta'',p'')=(\mathpzc{E}_1,D_1)+(\mathpzc{E}_3,D_3).$$
\end{theorem}

\begin{proof}
Recall the cylinder construction of Lemma \ref{cylinderconstruction}. We construct $(\mathcal{F}''_\infty,Q''_\infty)$ as above but by gluing the symmetric chain $(\mathcal{F},Q)\oplus -(\mathpzc{E}_{3,\cyl}, D_{3,\cyl})$ together with $(\mathcal{F}',Q')\oplus -(\mathpzc{E}_{1,\cyl}, D_{1,\cyl})$ along their common boundary $(\mathpzc{E}_1,D_1)\oplus (\mathpzc{E}_2,D_2)\oplus(\mathpzc{E}_3,D_3)$. It follows from \cite[Proposition 7.3]{hilsumcmodbun} that $Q''_\infty$ is self-adjoint and regular. An algebraic manipulation shows that $(\mathpzc{F}'', Q'')$ is symmetric and regular. It is immediate from the construction that $(\mathpzc{F}'', Q'')$ is a symmetric chain with boundary $(\mathpzc{E}_1,D_1)+(\mathpzc{E}_3,D_3)$ relative to the boundary data $(\theta'',p'')$.

It remains to prove that $(\mathpzc{F}'', Q'')$ is a half-closed $(C^\infty_c((0,1],\mathcal{A}),B)$-chain. (In fact, we also provide an alternative proof for the regularity of $Q''$). We will construct a locally compact resolvent for $\tilde{Q}''$ up to compact error terms. For $\lambda\neq 0$,  we define the operator $r_\lambda$ on $\mathcal{F}''$ through
$$r_\lambda:=b_c(\sqrt{\chi_1})(\lambda i+\tilde{Q})^{-1}b_c(\sqrt{\chi_1})+b_c'(\sqrt{\chi_2})(\lambda i+\tilde{Q}')^{-1}b_c'(\sqrt{\chi_2}),$$ 
where $b_c(\phi)=b(\phi)p_2+\phi(1)(1-p_2)$ for $b$ constructed from $(\mathcal{F},Q,\theta,p)$ and $b_c'$ is defined analogously using $(\mathcal{F}',Q',\theta',p')$. A crucial property of these representations is that 
\begin{equation}
\label{prodrep}
b_c(C^\infty_c(0,1])b_c'(C^\infty_c(0,1])\subseteq b(C_c^\infty(0,1)) b'(C_c^\infty(0,1)).
\end{equation}
By construction, $r_\lambda$ is an adjointable operator on $\mathcal{F}''$ which is locally compact for the  $C^\infty_c((0,1],\mathcal{A})$-action on $\mathcal{F}''$ defined from the boundary data $(\theta'',p'')$. We will proceed to verify that $(\lambda i+\tilde{Q}'')r_\lambda-1$ is compact with $\|(\lambda i+\tilde{Q})r_\lambda-1\|_{\End^*_B(\mathcal{F}'')}<1$ for $\lambda$ large enough. This shows that $(\mathpzc{F}'', \lambda^{-1} Q'')$ is half-closed for $\lambda>0$ large enough, hence it holds for any $\lambda$. Note that $r_\lambda(\mathcal{F}'')\subseteq \Dom(\tilde{Q}'')$. Hence 
\begin{align*}
(\lambda i+\tilde{Q}'')r_\lambda=&(\lambda i+\tilde{Q})b_c(\chi_1^{3/2})(\lambda i+\tilde{Q})^{-1}b_c(\sqrt{\chi_1})\\
&+(\lambda i+\tilde{Q}')b_c'(\chi_2^{3/2})(\lambda i+\tilde{Q}')^{-1}b_c'(\sqrt{\chi_2})\\
&+(\lambda i+\tilde{Q}')\cdot\!\!\!\!\!\!\underbrace{b_c'(\chi_2)b_c(\chi_1^{1/2})}_{\in b(C_c^\infty(0,1)) b'(C_c^\infty(0,1))}\!\!\!\!\!\!\cdot\, (\lambda i+\tilde{Q})^{-1}b_c(\sqrt{\chi_1})\\
&+(\lambda i+\tilde{Q})\cdot\!\!\!\!\!\!\underbrace{b_c'(\chi_2^{1/2})b_c(\chi_1)}_{\in b(C_c^\infty(0,1)) b'(C_c^\infty(0,1))}\!\!\!\!\!\!\cdot\, (\lambda i+\tilde{Q}')^{-1}b_c(\sqrt{\chi_2})\\
&=1+b_c(\phi_1)(\lambda i+\tilde{Q})^{-1}b_c(\sqrt{\chi_1})+b_c'(\phi_2)(\lambda i+\tilde{Q}')^{-1}b_c'(\sqrt{\chi_2}),
\end{align*}
where $\phi_1,\phi_2\in C^\infty_c(0,1)$ are given by $\phi_1:=(\chi_1^{3/2})'+(\chi_1^{1/2}\sigma^*\chi_2)'$ and $\phi_2:=(\chi_2^{3/2})'+(\chi_2^{1/2}\sigma^*\chi_1)'$ and $\sigma(t):=1-t$. These computations imply that
$$\|(\lambda i+\tilde{Q})r_\lambda-1\|_{\End^*_B(\mathcal{F}'')}=\mathcal{O}(|\lambda|^{-1}), \quad\mbox{as $|\lambda|\to \infty$},$$
as required.

\end{proof}

\begin{cor}
\label{transitivitycoror}
If $(\mathpzc{E}_1,D_1)\sim_{\bor}(\mathpzc{E}_2,D_2)$ and $(\mathpzc{E}_2,D_2)\sim_{\bor}(\mathpzc{E}_3,D_3)$, then  $(\mathpzc{E}_1,D_1)\sim_{\bor}(\mathpzc{E}_3,D_3)$. 
\end{cor}

\subsection{The bordism group}

We now turn to the equivalence classes of closed cycles determined by the bordism relation. The results of the last two sections imply that it is indeed an equivalence relation:

\begin{prop}
\label{bordismequi}
The bordism relation is an additive equivalence relation on the semigroup of isomorphism classes of closed $(\mathcal{A},B)$-cycles.
\end{prop}

\begin{proof}
The bordism relation is symmetric by Corollary \ref{symmetryofbordism}, reflexive by Lemma \ref{reflexivityofbordism} and transitive by  Corollary \ref{transitivitycoror}. The relation is additive by Proposition \ref{additivityofcor}.
\end{proof}

\begin{define} 
\label{defOfBorGrp}
We define the \emph{bordism group} $\Omega_*(\mathcal{A},B)$ as the set of bordism classes of closed $(\mathcal{A},B)$-cycles.
\end{define}

Lemma \ref{semiisoclass} and Proposition \ref{bordismequi} imply that $\Omega_*(A,B)$ forms a well defined $\Z/2\Z$-graded abelian semigroup under direct sum graded by parity of cycles.

\begin{theorem}
\label{bddtrasnsurj}
The abelian semigroup $\Omega_*(\mathcal{A},B)$ forms a $\Z/2\Z$-graded abelian group and the bounded transform 
$$b:\Omega_*(\mathcal{A},B)\to KK_*(A,B), \quad b[\mathpzc{E},D]:=[\mathpzc{E},b(D)],$$
is a well defined group homomorphism. Moreover, given two separable $C^*$-algebras $A$ and $B$, there exists a dense $*$-subalgebra $\mathcal{A}\subseteq A$ with a complete locally convex topology such that $b:\Omega_*(\mathcal{A},B)\to KK_*(A,B)$ is surjective. If $KK_*(A,B)$ is countably generated, $\mathcal{A}$ can be taken to be a Fr\'echet algebra and if $KK_*(A,B)$ is finitely generated, $\mathcal{A}$ can be taken to be a Banach algebra.
\end{theorem}

\begin{proof}
Any element of $\Omega_*(A,B)$ possesses an inverse by Corollary \ref{symmetryofbordism}. The bounded transform respects bordism by Remark \ref{bordismandbddtrans} (see \cite[Theorem 6.2]{hilsumbordism}). Existence of the dense subalgebra $\mathcal{A}$ follows from Theorem \ref{surjoncycles}.
\end{proof}

\section{The bounded transform} 
\label{boundSection}

In this section we go deeper into the mapping properties of the bounded transform. Kasparov's $KK$-theory is built out of bounded $KK$-cycles \cite{Kas}, i.e., pairs $(\mathpzc{E},F)$ as above only with $F$ being an adjointable operator that modulo the space of locally compact operators is a self-adjoint symmetry. To understand mapping properties of the bounded transform, one needs not only to understand how cycles relate to each other, but foremost how the corresponding relations relate to each other. The relation on the set of bounded $KK$-cycles can for instance be constructed by \emph{generating} a relation from continuous homotopies of the operator $F$ and identifying the so called degenerate cycles with the trivial cycle, see more in for instance \cite{Bla, Kas,jeto}. 

\subsection{Degenerate cycles}

We will start by considering a property analogous to the notion of degenerate unbounded $KK$-cycles (see \cite[Appendix by Kucerovsky]{valettebook}). Recall that a bounded $(A,B)$-Kasparov cycle $(\mathpzc{E},F)$ is called degenerate if
$$\pi(a)(F^2-1)=\pi(a)(F^*-F)=[\pi(a),F]=0 \quad \forall a\in A.$$
Note that if $A=\C$ acts unitally on $\mathpzc{E}$, the degeneracy condition simplifies substantially to the condition $F^2-1=F^*-F=0$.

Although it might be difficult to see at first, the next definition is inspired by a construction in geometric $K$-homology, see for example \cite[proof of Proposition 6.6]{HReta}, \cite[Section 2.3]{deelgoffIII} or below in Lemma \ref{vectorbmodgeo} of Subsection \ref{secassem}.

\begin{define}
\label{weakdegdeg}
We say that a closed $(\mathcal{A},B)$-cycle $(\mathpzc{E},D)$ is \emph{weakly degenerate} if we have decomposition $D=D_0+S$ where
\begin{enumerate}
\item $D_0$ and $S$ are self-adjoint regular operators (both odd if $\mathpzc{E}$ is nontrivially graded) satisfying 
$$\Dom D=\Dom S\cap \Dom D_{0}.$$
\item $S$ admits a bounded adjointable inverse, its domain is preserved by $A$ and $S$ commutes with the $A$-action.
\item There is a common core $X_{0}\subset \Dom S\cap \Dom D_{0}$ for which $SX_{0}\subset \Dom D_{0}$, $D_{0}X_{0}\subset \Dom S$ and 
$$D_0 S +SD_0=0 \quad\mbox{on}\quad X_{0}.$$
\item $\pi$ restricts to a continuous homomorphism $\pi:\mathcal{A}\to \Lip(D_0)$.
\end{enumerate}
If $(\mathpzc{E},D)$ is weakly degenerate with $D_0=0$, then $(\mathpzc{E},D)$ is called \emph{degenerate} (see for example \cite[Appendix by Kucerovsky]{valettebook}). 
\end{define}

It is immediate from the definition that (if one picks the correct chopping function) the bounded transform of a degenerate unbounded $KK$-cycle is a degenerate bounded $KK$-cycle. We can prove the nullbordance of weakly degenerate cycles using a swindle.

\begin{theorem}
\label{weakdegnullbor}
If $(\mathpzc{E},D)$ is weakly degenerate, there is a closed cycle $(\mathpzc{E}_\infty,D_\infty)$ such that $(\mathpzc{E}_\infty,D_\infty)+(\mathpzc{E},D)\sim_{\bor} (\mathpzc{E}_\infty,D_\infty)$. In particular, weakly degenerate cycles are null bordant.
\end{theorem}

\begin{proof}
We define $\mathpzc{E}_\infty:=\ell^2(\mathbb{N}_{\geq 2})\otimes \mathpzc{E}$ and 
$$D_\infty:=\bigoplus_{k=2}^\infty (D_0+kS),$$
whose domain is defined from the graph closure of $C_c(\mathbb{N}_{\geq 2}, \Dom D)$. Let $\pi_\infty$ denote the left $A$-action on $\mathpzc{E}_\infty$. Since 
$$[D_\infty,\pi_\infty(a)]=\bigoplus_{k=2}^\infty [D_0,\pi(a)],$$
we have that $\pi_\infty(\mathcal{A})\subseteq \Lip(D_\infty)$. The operator $D_\infty$ has locally compact resolvent because 
$$(1+D_\infty^2)^{-1}=\bigoplus_{k=2}^\infty (1+D_0^2+k^2S^2)^{-1},$$
and we have $\|(1+D_0^2+k^2S^2)^{-1}\|_{\End^*_B(\mathpzc{E})}=O(k^{-1})$, because $S$ admits a bounded adjointable inverse.

Take an increasing function $\chi\in C^\infty[0,1]$ such that $\chi(t)=0$ near $t=0$ and $\chi(t)=1$ near $t=1$. Consider the path of operators given by $\Dom D(t)=\Dom D_\infty$ and 
$$D(t):=\bigoplus_{k=2}^\infty (D_0+(k-\chi(t))S).$$
It is clear that $D(0)=D_\infty$ and 
$$(\mathpzc{E}_\infty,D(1))=\left(\mathpzc{E}_\infty, \bigoplus_{k=2}^\infty(D_{0}+(k-1)S)\right)\cong (\mathpzc{E}_\infty\oplus \mathpzc{E},D_\infty\oplus D).$$
It follows from Lemma \ref{difftopyofoperator} that $(\mathpzc{E}_\infty,D_\infty)+(\mathpzc{E},D)\sim_{\bor} (\mathpzc{E}_\infty,D_\infty)$. Therefore, in the abelian group $\Omega_*(\mathcal{A},B)$, we have
$$[\mathpzc{E},D]=[\mathpzc{E}_\infty,D_\infty]-[\mathpzc{E}_\infty,D_\infty]=0.$$
\end{proof}

We note that the same proof applies to the following situation.

\begin{prop}
\label{degnulb}
Let $(\mathpzc{E},D)$ by an $(\mathcal{A},B)$-cycle such that $\mathcal{A}\mathpzc{E}=0$. Then $(\mathpzc{E},D)$ is null-bordant.
\end{prop}

\begin{proof}
Consider the module $\mathpzc{E}_\infty:=\ell^2(\mathbb{N}_{\geq 2})\otimes \mathpzc{E}$ and the operator
$$D_\infty:=\bigoplus_{k=2}^\infty kD,$$
which we equip with the domain given by the graph closure of $C_c(\mathbb{N}_{\geq 2}, \Dom D)$. Since $\A$ acts as $0$ on $\mathpzc{E}$, $D_\infty$ has vanishing commutators with $\A$ and locally compact resolvent. The proof now proceeds as in Theorem \ref{weakdegnullbor}.
\end{proof}

Let us give a second proof of the fact that weakly degenerate cycles are null bordant. This proof is more geometric, retaining the geometric flavor of classical bordisms.

\begin{theorem}
Let $(\mathpzc{E},D)$ be a weakly degenerate cycle. Choose a function $\chi\in C^\infty(\R,\R_{>0})$ such that $\chi(t)=1$ for $|t|\leq 2$ and $\chi(t)=\sqrt{1+t^2}$ for $|t|$ large and denote the associated self-adjoint regular operator on $\mathpzc{E}_{\cyl}$ by $X$. Define $D_{\textnormal{cone}}$ to be the graph closure of 
\begin{equation}
\label{diffexpdefdcone}
\begin{cases}
&i\gamma_{\mathpzc{E}}\left(\frac{\partial}{\partial t}+D_0+XS\right), \\
&\qquad\qquad\qquad\mbox{on}\quad C^\infty_c((0,\infty),\Dom D)\quad\mbox{in the even case};\\
\,\\
&\begin{pmatrix}
0& \frac{\partial}{\partial t}+D_0+XS\\
-\frac{\partial}{\partial t}+D_0+XS& 0
\end{pmatrix}
\\&\qquad\qquad\qquad\mbox{on}\quad C^\infty_c((0,\infty),\Dom D\oplus \Dom D)\quad\mbox{in the odd case}.
\end{cases}.
\end{equation}
Then, $(\mathpzc{E}_{\cyl},D_{\textnormal{cone}},\theta_{\cyl},p_{\cyl})$ is a bordism and
\[\partial(\mathpzc{E}_{\cyl},D_{\textnormal{cone}},\theta_{\cyl},p_{\cyl})=(\mathpzc{E},D).\]
In particular, weakly degenerate cycles are nullbordant.
\end{theorem}

\begin{proof}
By similar considerations as in Subsection \ref{extprodofsymsect}, with Theorem \ref{sumregul} playing the main r\^ole, the domain of the operator $D_{\textnormal{cone}}$ in the even case is
$$\Dom(D_{\textnormal{cone}})=H^1_0(\R_+,\mathpzc{E})\cap X^{-1}L^2(\R_+,\Dom(S))\cap L^2(\R_+,\Dom(D_0)).$$
A similar expression holds in the odd case. We also have that $D_{\textnormal{cone}}^*$ is the differential expression \eqref{diffexpdefdcone} with the core
$$H^1(\R_+,\mathpzc{E})\cap X^{-1}L^2(\R_+,\Dom(S))\cap L^2(\R_+,\Dom(D_0)),$$
in the even case and a similar expression in the odd case. This proves that $(\mathpzc{E}_{\cyl},D_{\textnormal{cone}})$ is a well defined symmetric $(\mathcal{A},B)$-cycle. It is clear from the construction that $(\mathpzc{E},D)$ is a boundary of $(\mathpzc{E}_{\cyl},D_{\textnormal{cone}})$ with boundary data $(\theta_{\cyl},p_{\cyl})$. 

It remains to prove that whenever $\phi\in C^\infty_c(0,1]$, $b(\phi)(1+D_{\textnormal{cone}}^*D_{\textnormal{cone}})^{-1}$ is compact. Equivalently, we prove that the composition 
\begin{equation}
\label{compactoinclos}
\Dom(D_{\textnormal{cone}})\xrightarrow{b(\phi)}\Dom(D_{\textnormal{cone}})\hookrightarrow L^2(\R_+,\mathpzc{E}),
\end{equation}
is compact. To do so, we use the extension to even functions $u:L^2(\R_+,\mathpzc{E})\to L^2(\R,\mathpzc{E})$. It is clear that $u$ induces an inner product preserving map 
$$H^1_0(\R_+,\mathpzc{E})\to H^1(\R,\mathpzc{E}).$$
Let $D_{\dc}$ denote the differential expression \eqref{diffexpdefdcone}, extended to $\R$ by making $X$ even and equipped with the domain
$$\Dom(D_{\dc}):=H^1(\R,\mathpzc{E})\cap X^{-1}L^2(\R,\Dom(S))\cap L^2(\R,\Dom(D_0)).$$
It follows from Theorem \ref{sumregul} that $D_{\dc}$ is self-adjoint. If $\Dom(D_{\textnormal{dc}})\hookrightarrow L^2(\R,\mathpzc{E})$ is compact, then the operator in Equation \eqref{compactoinclos} is compact for all $\phi\in C^\infty_c(0,1]$. 

To prove that $D_{\dc}$ has compact resolvent, we consider the differential expression 
$$H:=-\frac{\partial^2}{\partial t^2}+D_0^2+(1+t^2)S^2,$$
which we equip with the domain 
$$\Dom H=H^2(\R,\mathpzc{E})\cap X^{-2}L^2(\R,\Dom S^2)\cap L^2(\R,\Dom D_0^2).$$
Then $D_{\dc}^2-H=\chi'S+(\chi^2(t)-1-t^2)S^2$ on $\Dom(D_{\dc}^2)\cap \Dom(H)$. As such, $D_{\dc}^2$ is a relatively bounded perturbation of $H$. Hence it suffices to prove that $H$ has compact resolvent.

Let $(h_n)_{n\in \field{N}}$ denote the Hermite functions. These are real-analytic functions such that $( -\frac{\partial^2}{\partial t^2}+t^2)h_n=(2n+1)h_n$ and $h_n(t)=\mathcal{O}(e^{-t^2/2})$ as $|t|\to \infty$. Consider the operator-valued integral kernel given through the norm-convergent sum
$$k(t,s)=\sum_{n=0}^\infty \left(2(n+1)+D_0^2+S^2\right)^{-1}h_n(|S|t)h_n(|S|s).$$
It is easily verified that $k\in L^2(\R\times \R,\K_B(\mathpzc{E}))$ and that if $f(t)=\int_\R k(t,s)g(s)\mathrm{d} s$, for $g\in L^2(\R,\mathpzc{E})$, then $(1+H)f=g$. We conclude that $(1+H)^{-1}\in \K_B(L^2(\R,\mathpzc{E}))$.
\end{proof}

\begin{remark}
For $\mathcal{A}\neq \C$, it is a non-trivial question whether a degenerate bounded $KK$-cycles lifts, modulo bordism, to a weakly degenerate unbounded $KK$-cycle. It is unclear to the authors what the answer to this question is. 
\end{remark}

\subsection{$K$-theory and the bordism group}

Our aim is to prove that the bounded transform $b:\Omega_*(\C,B)\to KK_*(\C,B)$ is an isomorphism for any $C^*$-algebra $B$. The key technical lemma is the following result on index theory of $B$-linear operators.

\begin{lemma}
\label{indzerolem}
Suppose that $D:\mathpzc{E}\dashrightarrow \mathpzc{F}$ is a regular operator which is Fredholm and $\ind_B(D)=0$ in $K_0(B)$. Then there is an $N\in \field{N}$ and a bounded adjointable operator $T\in \Hom_B^*(\mathpzc{E}\oplus B^N,\mathpzc{F}\oplus B^N)$ such that $D\oplus 0_N+T$ is invertible. If $\mathpzc{E}=\mathpzc{F}$ and $D$ is self-adjoint with $\ind_B(D)=0$ in $K_1(B)$ we can take $T$ self-adjoint.
\end{lemma}

\begin{proof}
We restrict to the even case, the odd case follows by standard Clifford algebra techniques. Recall that the index of $D$, when viewing $D$ as a Fredholm operator $D:W(D)\to \mathpzc{F}$, can be constructed using the amplification techniques of \cite{Exel}. Let $F:=D(1+D^{*}D)^{-\frac{1}{2}}$ be the bounded transform of $D$. By \cite[Lemma 3.8]{Exel} there exists $N_0\in \N$ and a finite rank operator $R:B^{N_0}\to \mathpzc{F}$ such that
\[F_{\amp}:=\begin{pmatrix} F & R \\ 0 & 0 \end{pmatrix}:\mathpzc{E}\oplus B^{N_0}\to  \mathpzc{F}\oplus B^{N_0},\]
has closed range. Since $\textnormal{im}\,F_{\amp}=\textnormal{im}\,F+\textnormal{im} \,R$ and $\textnormal{im}\, F=\textnormal{im}\,D$, it follows that 
\begin{align*}
D_{\amp}:=\begin{pmatrix} D & R \\ 0 & 0 \end{pmatrix}:&\,\mathpzc{E}\oplus B^{N_0}\dashrightarrow \mathpzc{F}\oplus B^{N_0},\\
&\mbox{with}\quad \Dom(D_{\amp}):=\Dom(D)\oplus B^{N_0},
\end{align*}
has closed range. The index is given by $\ind_B(D):=[\ker D_{\amp}]-[\ker D_{\amp}^*]$, and is independent of the choice of $D_{\amp}$ by \cite[Proposition 3.11]{Exel}. If $\ind_B(D)=0$, there is a unitary isomorphism of closed complemented submodules  \[T_0:\ker D_{\amp}\oplus B^{N_1}\to \ker D_{\amp}^*\oplus B^{N_1}\] for some $N_1$. Set $N:=N_0+N_1$ and consider $D_{\amp}\oplus 0:\mathpzc{E}\oplus B^{N}\dashrightarrow \mathpzc{F}\oplus B^{N}$ which we denote again by $D_{\amp}$ and $T_{0}:\ker D_{\amp}\xrightarrow{\sim} \ker D_{\amp}^{*}$. Since $D_{\amp}$ has closed range  we can write
\[\mathpzc{E}\oplus B^{N}\cong \ker D_{\amp}\oplus \textnormal{im} \,D_{\amp}^{*}, \quad \mathpzc{F}\oplus B^{N}\cong \ker D_{\amp}^{*}\oplus \textnormal{im}\, D_{\amp}.\]
Therefore the two mappings 
\begin{align*}
D_{\amp}:\textnormal{im}\, D_{\amp}^{*}&\cap \,\Dom D_{\amp}\to\textnormal{im} \,D_{\amp}\\
&\mbox{and}\quad D_{\amp}^{*}: \Dom D_{\amp}^{*}\cap\, \textnormal{im}\, D_{\amp}\to \textnormal{im}\,D_{\amp}^{*},
\end{align*} 
are  bijections. We can extend $T_0$ to an adjointable operator $T_1:\mathpzc{E}\oplus B^N\to \mathpzc{F}\oplus B^N$ via
\[T_1: \mathpzc{E}\oplus B^{N}\xrightarrow{\sim} \ker D_{\amp}\oplus \textnormal{im}\,D_{\amp}^{*}\xrightarrow{T_{0}\oplus 0} \ker D_{\amp}^{*}\oplus \textnormal{im}\, D_{\amp}\xrightarrow{\sim}\mathpzc{F}\oplus B^{N}.\]
The operators $D_{\amp}+T_1$ and $D_{\amp}^*+T^*_1$ are by construction invertible finite rank perturbations of $D\oplus 0$ proving the Lemma.
\end{proof}

\begin{theorem}
\label{bddtrandk} 
The bounded transform $b:\Omega_*(\C,B)\to KK_*(\C,B)$ is an isomorphism.
\end{theorem}

\begin{proof}
By Theorem \ref{bddtrasnsurj}, $b:\Omega_*(\C,B)\to KK_*(\C,B)$ is surjective. Again we restrict to the even case. Suppose that $[\mathpzc{E},D]\in \ker b$. After applying the index mapping $\ind_B:KK_0(\C,B)\to K_0(B)$ it follows that $\ind_B(D)=0$ in $K_0(B)$. By Corollary \ref{bddpert} and Lemma \ref{indzerolem}, there is a bordism $(\mathpzc{E},D)\sim_{\bor}(\mathpzc{E}\oplus B^N\oplus -B^N,D\oplus 0_{2N}+T)$ where $D\oplus 0_{2N}+T$ is invertible. If $D\oplus 0_{2N}+T$ is invertible, $(\mathpzc{E},D\oplus 0_{2N}+T)$ is a degenerate $(\C,B)$-cycle and Theorem \ref{weakdegnullbor} implies that $(\mathpzc{E},D)$ is nullbordant. Hence $\ker b=0$.
\end{proof}

\section{Analytic assembly on manifolds}\label{BDBorKKSection}

The bounded transform admits a splitting when $\mathcal{A}$ is the Lipschitz algebra of a Riemannian manifold $X$. This splitting involves solving a Baum-Douglas type index problem, see \cite{BD,BD2}. Let us first recall the cycles in the Baum-Douglas model for $K$-homology. Throughout the section, $B$ will denote a fixed unital $C^*$-algebra.

\subsection{Baum-Douglas models for $K$-homology}

The topic of this section is closely related to the examples of Subsection \ref{mfdwbdryex}. 

\begin{define} \label{BDcycle}
Let $X$ be a compact topological space. A \emph{geometric cycle} over $X$ with coefficients in $B$ is a triple $(M,E_B,f)$ where
\begin{enumerate}
\item $M$ is a closed spin$^c$-manifold;
\item $E_B\to M$ is a Hermitian $B$-bundle;
\item $f:M\to X$ is a continuous mapping.
\end{enumerate}
If $M$ is a Riemannian manifold and $\nabla_E$ is a $B$-linear Hermitian connection on $E_B$, we say that $(M,E_B,\nabla_E,f)$ is a \emph{Riemannian cycle with connection}.

An isomorphism between two geometric cycles $(M,E_B,f)$ and $(M',E_B',f')$ is a diffeomorphism $u:M'\to M$ lifting to a unitary isomorphism $u^*E_B\cong E_B'$ of Hermitian $B$-bundles satisfying that $f\circ u=f'$. An isomorphism between two Riemannian cycles with connection $(M,E_B,\nabla_E,f)$ and $(M',E_B',\nabla_{E'},f')$ is an isometric isomorphism $u$ satisfying the conditions above and in addition the isomorphism $u^*E_B\cong E_B'$ preserves the connections.
\end{define}

We note that the manifold $M$ need not be connected nor have the same dimensionality of the components. If the dimensionality of the components in $(M,E_B,f)$ is even/odd we say that this cycle has even/odd parity. For $*=$even/odd, we let $BD_*(X;B)$ denote the set of isomorphism classes of geometric cycles of parity $*$. This set forms an abelian semigroup under disjoint union:
$$(M,E_B,f)\dot{\cup}(M',E_B',f'):=(M\dot{\cup}M',E_B\dot{\cup}E_B',f\dot{\cup}f').$$
Similarly, we let $\widetilde{BD}_*(X;B)$ denote the abelian semigroup of isomorphism classes of Riemannian cycles with connection of parity $*$.

If $X$ is a Riemannian manifold, we let $BD_{*,\Lip}(X;B)\subseteq BD_*(X;B)$ denote the subsemigroup of isomorphism classes of cycles $(M,E_B,f)$ with $f:M\to X$ being Lipschitz in \emph{some} Riemannian structure on $M$ (compactness of $M$ guarantees that in this case it is Lipschitz in all Riemannian structures on $M$). We let $\widetilde{BD}_{*,\Lip}(X;B)$ denote the set of isomorphism classes of Riemannian cycles with connection $(M,E_B,\nabla_E,f)$ with $f:M\to X$ being Lipschitz.

\begin{define}
A \emph{geometric bordism} is a triple $(W,F_B,g)$ satisfying all the conditions of a geometric cycle, but $W$ is allowed to have a boundary. We write
$$\partial(W,F_B,g)=(\partial W,F_B|_{\partial W},g|_{\partial W}).$$ 
A \emph{Riemannian bordism} is a quadruple $(W,F_B,\nabla_F,g)$ satisfying all the conditions of a Riemannian cycle, but $W$ is allowed to have a boundary. We write
$$\partial(W,F_B,\nabla_F,g)=(\partial W,F_B|_{\partial W},\nabla_F|_{\partial W},g|_{\partial W}).$$ 
If a geometric cycle (Riemannian cycle) is the boundary of a bordism, we say it is \emph{nullbordant}.
\end{define}

\begin{define}
Suppose that $(M,E_B,\nabla_E,f)$ is a Riemannian cycle with connection and that $V\to M$ is a Hermitian spin$^c$-vector bundle with even-dimensional fibers. The \emph{vector bundle modification} of $(M,E_B,\nabla_E,f)$ along $V$ is defined as 
$$(M,E_B,\nabla_E,f)^V:=(M^V,E_B^V,\nabla_E^V,f^V),$$
where 
\begin{enumerate}
\item $M^V:=S(V\oplus 1_\R)$ is the sphere bundle in $V\oplus 1_\R$, where $1_\R\to M$ is the trivial real line bundle, with its induced spin$^c$-structure and Riemannian metric;
\item $f^V:M^V\to X$ is the composition of the projection $\pi_V:M^V\to M$ with $f$;
\item $E_B^V:=\pi_V^*E_B\otimes Q_V$ where $Q_V\to M^V$ is the Bott bundle (see \cite[Section 2.5]{Rav});
\item $\nabla_E^V=\pi_V^*(\nabla_E)\otimes \nabla_{Q_V}$ for the standard connection $\nabla_{Q_V}$ on the Bott bundle.
\end{enumerate}
If $(M,E_B,f)$ is a cycle and $V\to M$ is a Hermitian spin$^c$-vector bundle with even-dimensional fibers, the vector bundle modification of $(M,E_B,f)$ along $V$ is defined as 
$$(M,E_B,f)^V:=(M^V,E_B^V,f^V),$$
\end{define}

\begin{define}
\label{bdrelations}
We define the equivalence relation $\sim$ on $BD_*(X;B)$ respectively $\widetilde{BD}_*(X;B)$ to be the one generated by bordism, vector bundle modification, and disjoint union/direct sum; the last of these relations is defined as follows:
\begin{align*}
\qquad\qquad (M,E_B,f)\dot{\cup}(M,&E_B',f)\sim (M,E_B\oplus E_B',f)\quad\mbox{respectively}\\ 
(M,E_B,&\nabla_E,f)\dot{\cup}(M,E_B',\nabla_E',f)\sim (M,E_B\oplus E_B',\nabla_E\oplus \nabla_E',f).
\end{align*}
\end{define}

It is clear that for a compact Riemannian manifold $X$, the equivalence relation restricts to a well defined relation also on $BD_{*,\Lip}(X;B)$ respectively $\widetilde{BD}_{*,\Lip}(X;B)$. By a perturbation argument, cf. \cite[Proof of Theorem 5.1]{Kescont}, the relation that $\sim$ induces on $BD_{*,\Lip}(X;B)$ respectively $\widetilde{BD}_{*,\Lip}(X;B)$ coincides with that generated by bordisms $(W,F_B,g)$ such that the mapping $g$ is Lipschitz.

\begin{prop}
\label{kstarprop}
The semigroups 
\begin{enumerate}
\item $BD_*(X;B)/\sim$;
\item $\widetilde{BD}_*(X;B)/\sim$;
\end{enumerate}
are isomorphic abelian groups. 
\end{prop}

We will denote the abelian group of Proposition \ref{kstarprop} by $K_{*}^{\geo}(X;B)$. 

\begin{proof}
The proof that these are all abelian groups is similar in the different cases. For example, in the case of $BD_*(X;B)/\sim$, one notes that 
$$\partial(M\times [0,1],E_B,f\circ \pi_M)=(M,E_B,f)\dot{\cup}(-M,E_B,f),$$
where $-M$ is $M$ equipped with the opposite spin$^c$-structure.

To prove that the groups are isomorphic, we show that
$$\widetilde{BD}_*(X;B)/\!\sim\;\longrightarrow BD_*(X;B)/\!\sim, \quad (M,E_B,\nabla_E,f)\mapsto (M,E_B,f), $$
is an isomorphism. Indeed, its inverse is given by $(M,E_B,f)\mapsto  (M,E_B,\nabla_E,f)$ for some choice of Riemannian metric on $M$ and connection on $E_B$, since the choices are made in a path connected space it is clearly well defined up to bordism.
\end{proof}

\begin{theorem}
\label{muaniso}
In the notation of Lemma \ref{firstgammadef} on page \pageref{firstgammadef}, The analytic assembly map is defined via
\begin{align*}
\mu_{an}:K_{*}^{\geo}(X;B)\to &KK_*(C_0(X),B),\\
& \mu_{\an}(M,E_B,\nabla_E,f):=f_*\left(\mathpzc{E}_M,D^M_E\left(1+(D^M_E)^2\right)^{-1/2}\right).
\end{align*}
It is a well defined group homomorphism. If $X$ is a finite $CW$-complex, $\mu_{\an}$ is an isomorphism.
\end{theorem}

For a proof of this theorem see for instance \cite[Chapter 2.3]{Wal}. It combined with Poincar\'e duality on smooth manifolds implies the following:

\begin{cor}
For a smooth Riemannian manifold $X$, the following semigroups are abelian groups isomorphic to the groups of Proposition \ref{kstarprop}:
\begin{enumerate}
\item[(3)] $BD_{*,\Lip}(X;B)/\sim$;
\item[(4)] $\widetilde{BD}_{*,\Lip}(X;B)/\sim$;
\end{enumerate}

\end{cor}
The reader should note the following assumption: {\bf for a smooth manifold N, we tacitly represent classes by cycles $\bf{(M,E_{B},\nabla_{E}, f)}$ for which the map $\bf{f:M\to N}$ is Lipschitz.}

\subsection{Analytic assembly and its consequences}
\label{secassem}

We now describe the geometric assembly map. This map allows us to compare the relations of geometric $K$-homology from Definition \ref{bdrelations} with the $KK$-bordism relation. We first describe the situation on the level of cycles, in order to pass to the bordism groups. Finally, we show that the bounded transform is split-surjective in the case of manifolds.

\begin{prop}
The constructions of Subsection \ref{mfdwbdryex}, see Lemma \ref{firstgammadef} on page \pageref{firstgammadef}, defines an additive map: 
$$\gamma_0:\widetilde{BD}_{*,\Lip}(X;B)\to Z_*(\Lip(X);B), \quad (M,E_B,\nabla_E,f)\mapsto f_*\gamma_0(M,E_B,\nabla_E).$$
\end{prop}

The definition of $\gamma_0$ can be found in Lemma \ref{firstgammadef} (see page \pageref{firstgammadef}). The proof of the next proposition follows immediately from Proposition \ref{honestbordismprop}.

\begin{prop}
\label{bordismgeocy}
Nullbordant Riemannian cycles are mapped to nullbordant cycles under $\gamma_0:\widetilde{BD}_{*,\Lip}(X;B)\to Z_*(\Lip(X);B)$.
\end{prop}

The next lemma provided the inspiration for the notion of weakly degenerate cycles of Definition \ref{weakdegdeg}, see page \pageref{weakdegdeg}.

\begin{lemma}
\label{vectorbmodgeo}
Whenever $(M,E_B,\nabla_E,f)$ is a Riemannian cycle with connection and $V\to M$ is a Hermitian spin$^c$-vector bundle with even-dimensional fibers, there is a weakly degenerate $(\Lip(X),B)$-cycle $(\mathpzc{E},D)$ such that 
$$\gamma_0((M,E_B,\nabla_E,f)^V)=\gamma_0(M,E_B,\nabla_E,f)+(\mathpzc{E},D).$$
\end{lemma}

The proof of this lemma is inspired by \cite{deelgoffIII,HReta}. 

\begin{proof}
By considering each of the connected components of $M$, we can assume that $V$ has constant rank. Let $2k$ denote the rank of $V$. We let $\pi_P:P\to M$ denote the principal $\Spin^c(2k)$-bundle of $\Spin^c(2k)$-frames of $V$. In particular
$$M^V=P\times_{\Spin^c(2k)}S^{2k}.$$
There is a bundle $Q_k\to S^{2k}$ with connection $\nabla_k$ such that $Q_V=P\times_{\Spin^c(2k)}Q_k$ and $\nabla_{Q_V}$ is induced from $\nabla_k$. We use $S_M$ to denote the spin$^c$-structure on $M$ and $S_{2k}$ to denote that on $S^{2k}$. Hence
\begin{equation}
\label{spinc2kinv}
L^2(M^V,E_B^V\otimes S_{M^V})=[L^2(P\times S^{2k},(E_B\otimes \pi_P^*S_M)\boxtimes (Q_k\otimes S_{2k}))]^{\Spin^c(2k)}.
\end{equation}
Here $[\mathcal{V}]^G$ denotes the $G$-invariant part of a $G$-representation $\mathcal{V}$. Since $\Spin^c(2k)$ is a compact group, the identification of $L^2(M^V,E_B^V\otimes S_{M^V})$ with the $\Spin^c(2k)$-invariant part of $L^2(P\times S^{2k},(E_B\otimes \pi_P^*S_M)\boxtimes (Q_k\otimes S_{2k}))$ gives a complemented submodule. 

Let $D_{2k}^Q$ denote the Dirac operator on $Q_k\otimes S_{2k}\to S^{2k}$, defined as the spin$^c$-Dirac operator twisted by $\nabla_k$. We also let $\epsilon_k$ denote the grading on $L^2(S^{2k},S_{2k}\otimes Q_k)$. By \cite[Lemma 6.7]{HReta}, $\ker D_{2k}^Q$ is a one-dimensional space spanned by an \emph{even} section; we let $\mathrm{e}_k\in \Psi^{-\infty}(S^{2k};S_{2k}\otimes Q_k)$ denote the projection onto the span of this section (i.e., the projection onto $\ker D_{2k}^Q$). The auxiliary operator
$$\tilde{D}^V:=\pi_P^*D^M\hat{\boxtimes}D^Q_{2k},$$
is densely defined on $C^\infty(P\times S^{2k},(E_B\otimes \pi_P^*S_M)\hat{\boxtimes} (Q_k\otimes S_{2k}))$. We remind the reader that the tensor products are graded. In the identification \eqref{spinc2kinv}, the Dirac operator on the vector bundle modification $(M,E_B,\nabla_E,f)^V$ is given by
$$D^{M^V}_{E^V}=\tilde{D}^V|_{L^2(M^V,E_B^V\otimes S_{M^V})}.$$
Since this is a Dirac operator on a closed manifold, it defines a selfadjoint regular operator by Lemma \ref{firstgammadef}. 

On the complemented $B$-submodule 
\begin{align*}
L^2(M,E_B\otimes S_M)&\cong (1_{L^2(M,E_B\otimes S_M)}\otimes \mathrm{e}_k)\cdot L^2(M^V,E_B^V\otimes S_{M^V})\\
&\subseteq L^2(P\times S^{2k},(E_B\otimes \pi_P^*S_M)\boxtimes (Q_k\otimes S_{2k})),
\end{align*}
We have that
$$D^{M^V}_{E^V}|_{L^2(M,E_B\otimes S)}=D^M_E.$$
Hence the restriction of the regular operator  $D^{M^V}_{E^V}$ to the complemented $B$-submodule $L^2(M,E_B\otimes S_M)$ coincides with the non-modified cycle. We use the following notation for the complement of this $B$-submodule and related operator:
\begin{align*}
\mathpzc{E}:=L^2(M,&\,E_B\otimes S_M)^\perp\\
&=(\id-1_{L^2(M,E_B\otimes S_M)}\otimes \mathrm{e}_k)\cdot L^2(M^V,E_B^V\otimes S_{M^V}),
\end{align*}
and 
$$D:=D^{M^V}_{E^V}|_{\mathpzc{E}\cap \Dom D^{M^V}_{E^V}}:\mathpzc{E}\dashrightarrow \mathpzc{E}.$$ 
Since $D$ is a direct summand in the self-adjoint regular operator $D^{M^V}_{E^V}$, $D$ is self-adjoint and regular. Moreover,
\begin{align*}
D^2=\left(\pi_P^*(D^M_E)^2\hat{\otimes} (\id-1_{L^2(M,E_B\otimes S_M)}\otimes \mathrm{e}_k) +1_{L^2(M,E_B\otimes S_M)}\hat{\otimes} (D_k^Q)^2\right)|_{\mathpzc{E}}.
\end{align*}
Since $(D_k^Q)^2\geq 1$ is strictly positive on the image of $\id-1_{L^2(M,E_B\otimes S_M)}\otimes \mathrm{e}_k$, it is invertible. Furthermore, $D_k^Q$ differentiates only in the fiber direction. An element $a\in C_b(X)$ acts as multiplication operator by the Lipschitz function $a\circ f\circ \pi_{M^V}\in \Lip(M^V)$ - a function constant on the fibers of $M^V\to M$. Hence $D_k^Q$ commutes with the action of $C_b(X)$ on $L^2(M^V,E_B^V\otimes S_{M^V})$. By construction $S:=D_k^Q|_{\mathpzc{E}}$ anticommutes with 
$$D_0:=\begin{cases}
(\pi_P^*(D^M_{E})\boxtimes \epsilon_{k}, \quad&\mbox{if $M$ is odd-dimensional};\\
(\pi_P^*(D^M_{E})\boxtimes  \mathrm{id}_{S^{2k}}, \quad&\mbox{if $M$ is even-dimensional}.
\end{cases}$$
It follows that $D=D_0+S$ is weakly degenerate and hence trivial in the bordism group. Hence, the proof is complete by observing that 
$$(L^2(M^V,E_B^V\otimes S_{M^V}),D^{M^V}_{E^V})=(L^2(M,E_B\otimes S_{M}),D^M_E)+(\mathpzc{E},D).$$
\end{proof}

We now come to the main result of this subsection, showing that the analytic assembly map $\mu$ factors through the bounded transform via the $KK$-bordism group. We will use this result to prove that the bounded transform is a split surjection in this particular case.

\begin{theorem}
\label{betamugamma}
If $X$ is a compact manifold with boundary, the map
$$\gamma:K^{{\rm geo}}_{*}(X;B) \to \Omega_*(\Lip(X),B)),\quad \gamma[M,E_B,\nabla_E,f]:=[\gamma_0(M,E_B,\nabla_E,f)],$$
is well defined and fits into a commuting diagram:
\begin{equation}
\label{commdiagramwithbdd}
\xymatrix{
 K^{{\rm geo}}_{*}(X;B) \ar[rdd]_{\mu}   \ar[rr]^{\gamma}  & & \Omega_*(\Lip(X),B))  \ar[ldd]^{b}  
 \\ \\
&KK_*(C(X),B)&
}
\end{equation}
\end{theorem}

\begin{proof}
It is a consequence of the construction that, on the level of cycles, the diagram \eqref{commdiagramwithbdd} commutes. Hence it will commute if the maps are well defined. It remains to prove that $\gamma$ is well defined. This follows since $\gamma$ clearly respects disjoint union/direct sum, respects the bordism relation by Proposition \ref{bordismgeocy} and respects vector bundle modification by Lemma \ref{vectorbmodgeo}.

\end{proof}

\begin{cor}
\label{splittingbeta}
If $X$ is a compact manifold with boundary, then the bounded transform $b:\Omega_*(\Lip(X),B)\to KK_*(C(X),B)$ is a split surjection.
\end{cor}

\begin{proof}
By Theorem \ref{muaniso}, $\mu_{\an}$ is an isomorphism and Theorem \ref{betamugamma} implies that $\gamma\circ \mu_{\an}^{-1}$ splits $b$.
\end{proof}

\begin{remark}
\label{nhremark}
A consequence of Corollary \ref{splittingbeta} is that for any compact manifold $X$ there is a direct sum decomposition 
$$\Omega_*(\Lip(X),B)\cong KK_*(C(X),B)\oplus NH_*(\Lip(X),B),$$
where any element $x\in NH_*(\Lip(X),B)$ has a nullhomotopic bounded transform but $x$ itself need not be nullbordant. 
\end{remark}

\section{Examples} \label{examplesSection}
\subsection{Examples of Hilsum bordisms} \label{exHilsumSubSec}
In this subsection we list various geometric situations which lead to Hilsum bordisms. All of these have appeared in Hilsum's previous work \cite{hilsumcmodbun, hilsumfol, hilsumbordism}, but it seems useful to summarize them here for the reader unfamiliar with his work. Furthermore, we note that there is overlap between the various examples. The first appears as a special case of each of the rest! Rather vaguely, the overall theme is that if a geometric situation gives an unbounded $KK$-cycle and the context provides a natural notion of bordism, then one often can obtain a Hilsum bordism.

\begin{ex} {\bf Compact manifolds} (see any of \cite{hilsumcmodbun, hilsumfol, hilsumbordism} and/or Subsection \ref{mfdwbdryex} and Proposition \ref{honestbordismprop}) \\
Let $W$ be a compact Riemannian manifold with boundary, $E$ a Hermitian vector bundle over $W$ and $f:W\to X$ a Lipschitz mapping. The operator $P$ is a formally self-adjoint elliptic pseudo-differential operator on $W$ of order $1$ acting on sections of $E$. We assume existence of a collaring $U\cong \partial W\times [0,1]$ near the boundary and that the data $E$ and $P$ respect this collar in the sense that there exists a boundary operator $P_\partial$ -- an elliptic self-adjoint operator of order $1$ on $\partial W$, such that for any $\chi\in C^\infty_c(U^\circ)$ it holds that $\chi P\chi=\chi\Psi(P_\partial)\chi$. Under these assumptions, one can form a Hilsum bordism with respect to $C^{\infty}(M) \subseteq C(M)$ from any closed symmetric extension of $P$. The associated boundary cycle is $(L^2(\partial W; E|_{\partial W}), P_{\partial})$. 

Special cases of this situation include the construction considered in Subsection \ref{mfdwbdryex} (that is, a Dirac operator associated to a ${\rm spin^c}$-structure twisted by a Hermitian vector bundle). In fact, if $B$ is a unital $C^*$-algebra, we saw (again in Subsection \ref{mfdwbdryex}) that by twisting by a Hermitian $B$-bundle one obtains a Hilsum bordism with respect to $C^{\infty}(M) \subseteq C(M)$ and $B$.
\end{ex} 

\begin{ex} {\bf Coverings} (see \cite[Section 8]{hilsumbordism}) \\
Let $\Gamma$ be a discrete group, $M$ an oriented Lipschitz manifold with a fixed (measurable) Riemannian metric, $E$ a Hermitian Lipschitz vector bundle over $M$, and $f: M \rightarrow B\Gamma$ a continuous map. Also, let $\tilde{M}$ denote the Galois covering obtained from $f$ and $\tilde{g}$ is the lift of a Riemannian metric on $M$ to $\tilde{M}$. Then, there is an associated unbounded $KK$-cycle, $(\mathcal{E}, A^E)$ coming from the signature operator, (leading to a class in $KK^*(C(M), C^*_r(\Gamma))$) where $C^*_r(\Gamma)$ is the reduced $C^*$-algebra of $\Gamma$.

Furthermore, following the setup of \cite[Theorem 8.4]{hilsumbordism}, suppose that $Z$ is a Lipschitz manifold with boundary isomorphic to $M$ with $F$ a Hermitian Lipschitz vector bundle over $Z$, and $g: Z \rightarrow B\Gamma$ a continuous map which extend respectively $E$ and $f$. Then, there is an associated Hilsum bordism, with respect to $\Lip(M) \subseteq C(M)$ and $C^*_r(\Gamma)$, which has boundary $(\mathcal{E}, A^E)$. 
\end{ex}

\begin{ex} {\bf Transversely elliptic operators} (see \cite[Section 9]{hilsumbordism}) \\
Let $G$ be a compact Lie group, $M$ be a compact $G$-manifold, $E$ be a Hermitian $G$-vector bundle over $M$, and $P$ be a pseudo-differential operator of order $1$ acting on sections of $E$. Assuming that $P$ is transversely elliptic, one obtains a class in $KK^*(C(M) \rtimes G, \C)$, which is realized by an unbounded $KK$-cycle, $x$.

Furthermore, assume that $W$ is a compact $G$-manifold with boundary, $F$ is a Hermitian $G$-vector bundle over $W$, and $Q$ a pseudo-differential operator order $1$ collared near the boundary as above such that $\partial W = M$, $F|_{\partial W}=E$, and $Q_\partial=P$. Then, there is an associated Hilsum bordism which has boundary $x$.
\end{ex}

\begin{ex} {\bf Orbifolds} \\
A special case of the previous example are orbifolds which are locally free quotients of a compact Lie group action. To see this, one works with the $G$-frame bundle (see \cite[Introduction]{Far}). The reader is directed to \cite[Remark 9.4]{hilsumbordism} and \cite{Far} for further details. 
\end{ex}
\begin{ex} {\bf Foliations} (see \cite[Section 5]{hilsumfol}) \\
Let $(M, \mathcal{F})$ be a closed foliated manifold. Then one can construct the holonomy groupoid, $\mathcal{G}$, with its associated reduced $C^*$-algebra; we denote this algebra by $C^*(M,\mathcal{F})$. If we assume that the restriction of the tangent bundle of $M$ to the foliation (i.e., $\mathcal{F}$) is spin$^c$, so we can form the longitudinal spin$^c$-Dirac operator and its associated unbounded $KK$-cycle and class in $KK^*(C^*(M,\mathcal{F}), \C)$.

Furthermore, assume that there is $(W, \mathcal{E})$ a compact foliated manifold with boundary such that $\partial W = M$ and $\mathcal{E}$ intersects $M$ transversely along $\mathcal{F}$. Also, assume that the restriction of the trangent bundle of $W$ to the foliation is ${\rm spin^c}$. Then, by fixing Riemannian metrics and a collaring, one can obtain a Hilsum bordism (with boundary the cycle produced in the previous paragraph); the precise construction is in \cite[Section 5]{hilsumfol}, see in particular \cite[Lemma 5.1]{hilsumfol}.
\end{ex}

\subsection{The bordism group for ideals of compact operators}
\label{compsubse}

Let $\mathcal{I}\subseteq \mathbb{B}(\mathpzc{H})$ be a symmetrically normed operator ideal. We assume that $\mathcal{I}$ is regular, i.e. the set of finite rank operators is dense in $\mathcal{I}$. We will compute a certain subgroup of $\Omega_*(\mathcal{I},B)$ and show that it is isomorphic to $K_*(B)$. We assume $\mathpzc{H}=\ell^2(\N)$ for simplicity.

\begin{define}
\label{esssubgroup}
We say that an $(\mathcal{A},B)$-cycle $(\mathpzc{E},D)$ is \emph{essential} if the $A$-action on $\mathpzc{E}$ is essential. We let $\Omega_*^e(\mathcal{A},B)$ denote the subgroup of $\Omega_*(\mathcal{A},B)$ generated by essential $(\mathcal{A},B)$-cycles.
\end{define}

\begin{prop}
Suppose that $\mathcal{A}$ is unital. Then $\Omega_*(\mathcal{A},B)=\Omega_*^e(\mathcal{A},B)$.
\end{prop}

\begin{proof}
We need to prove that any cycle is bordant to an essential cycle. Suppose that $(\mathpzc{E},D)$ is an $(\mathcal{A},B)$-cycle and let $p$ be the image in $\End_B^*(\mathpzc{E})$ of $1\in \mathcal{A}$. By construction, the projection $p$ commutes with the $A$-action and preserves the domain of $D$. Since $\mathcal{A}$ has bounded adjointable commutators with $D$, Corollary \ref{bddpert} provides us with a bordism
$$(\mathpzc{E},D)\sim_{\bor} (p\mathpzc{E},pDp)\oplus ((1-p)\mathpzc{E},(1-p)D(1-p)).$$
It holds that $p\mathpzc{E}=\overline{\mathcal{A}\mathpzc{E}}$ and $\mathcal{A}$ acts trivially on $(1-p)\mathpzc{E}$. Therefore the cycle $((1-p)\mathpzc{E},(1-p)D(1-p))$ is nullbordant by Proposition \ref{degnulb}. From the transitivity of the bordism relation, it follows that $(\mathpzc{E},D)\sim_{\bor} (p\mathpzc{E},pDp)$ and the latter cycle is essential.
\end{proof}

It remains an open question if $\Omega_*(\mathcal{A},B)=\Omega_*^e(\mathcal{A},B)$ in general. In the bounded setting, it was proved in \cite[Lemma 2.8]{Kas} that the class of any $KK$-cycle in the relevant $KK$-group always can be represented by an essential cycle. The statement in bounded $KK$-theory will have an immediate consequence in the $KK$-bordism groups {\bf only} for algebras $\mathcal{A}$ such that the bounded transform is injective.

\begin{prop}
Assume that $\mathcal{I}$ is a regular symmetrically normed ideal. The mapping 
$$\alpha:\Omega_*(\C,B)\to \Omega_*^e(\mathcal{I},B),\quad(\mathpzc{E},D)\mapsto (\mathpzc{E}\otimes\ell^2(\N),D\otimes\id_{\ell^2(\N)}),$$ 
is an isomorphism with inverse mapping being the restriction of the mapping
$$\beta_0:\Omega_*(\mathcal{I},B)\to \Omega_*(\C,B), \quad \beta_0(\tilde{\mathpzc{E}},\tilde{D}):=(e\tilde{\mathpzc{E}},e\tilde{D}e),$$
where $e\in \mathcal{I}$ is a rank one projection.
\end{prop}

\begin{proof}
It holds already at the level of cycles that 
$$\beta_0\circ\alpha(\mathpzc{E},D)=\left(e(\mathpzc{E}\otimes\ell^2(\N)),e(D\otimes\id_{\ell^2(\N)})e\right)\cong (\mathpzc{E},D),$$
so $\beta_0\circ\alpha=\id_{\Omega_*(\C,B)}$. The proof of the proposition is completed upon proving surjectivity of $\alpha$, that is, for any essential $(\mathcal{I},B)$-cycle $(\tilde{\mathpzc{E}},\tilde{D})$, there is a $(\C,B)$-cycle $(\mathpzc{E},D)$ and a bordism of $(\mathcal{I},B)$-cycles
\begin{equation}
\label{surbor}
(\mathpzc{E}\otimes\ell^2(\N),D\otimes\id_{\ell^2(\N)})\sim_{\bor}(\tilde{\mathpzc{E}},\tilde{D}).
\end{equation}

We note that since any element of $\mathcal{I}$ has bounded adjointable commutators with $\tilde{D}$, the continuity of $\mathcal{I}\to \Lip(\tilde{D})$ ensures the existence of a $C>0$ such that
\begin{equation}
\label{comandlip}
\|[\tilde{D},a]\|_{\End^*_B(\tilde{\mathpzc{E}})}\leq C\|a\|_\mathcal{I}.
\end{equation}
We take matrix units $(e_{jk})_{j,k=0}^\infty$. Its linear span is dense in $\K(\mathpzc{H})$ in the operator norm and by assumption dense in $\mathcal{I}$ in its norm. We also note that since $\tilde{\mathpzc{E}}$ is essential,
$$\tilde{\mathpzc{E}}=\bigoplus_{j=0}^\infty \mathpzc{E}_j, \quad\mbox{for}\quad \mathpzc{E}_j:=e_{jj}\tilde{\mathpzc{E}}.$$
Moreover, each $e_{jk}$ is a partial isometry with source $\mathpzc{E}_k$ and range $\mathpzc{E}_j$, i.e. it restricts to a unitary isomorphism $e_{jk}:\mathpzc{E}_k\to \mathpzc{E}_j$. 

We define the symmetric operator
$$T:=\sum_k (1-e_{kk})\tilde{D}e_{kk}.$$
It is clear that $T$ is relatively bounded to $\tilde{D}$ with relative bound $1$. It holds for any $e_{jk}$ that $e_{jk}T$ and $Te_{jk}$ are bounded adjointable and the density of finite rank operators combined with the estimate \eqref{comandlip} implies that $T$ is locally bounded. Write $D_0$ for the operator $\tilde{D}-T=\sum_{j\in \N} e_{jj} \tilde{D}e_{jj}$ extended to its graph closure of $\Dom(\tilde{D})$. Since $D_0$ is a relatively bounded perturbation of $\tilde{D}$ with relative bound $1$, W\"ust's extension of the Kato-Rellich theorem (see \cite[Theorem 4.6]{leschkaad2}) implies that $D_0$ is self-adjoint and regular. For any $a\in \mathcal{I}$, the local boundedness of $T$ shows that
\begin{equation}
\label{domainsarelikecrazynice}
a\Dom(D_0)\subseteq a\Dom(\tilde{D}),
\end{equation}
It follows that $\mathcal{I}$ preserves $\Dom(D_0)$. The algebra $\mathcal{I}$ has bounded adjointable commutators with $D_0$ because $[D_0,a]:\Dom(D_0)\to \tilde{\mathpzc{E}}$ is the closure of the operator $[\tilde{D}-T,a]:\Dom(\tilde{D})\to \tilde{\mathpzc{E}}$ in the graph topology of $\Dom(D_0)$ and the latter admits a bounded adjointable extension. Moreover, \eqref{domainsarelikecrazynice} and the compactness of the inclusion $a\Dom(\tilde{D})\subseteq \tilde{\mathpzc{E}}$ for any $a\in \mathcal{I}$ implies locally compact resolvent of $D_0$. It follows that $(\tilde{\mathpzc{E}},D_0)$ is an $(\mathcal{I},B)$-cycle. 

Take a function $\chi\in C^\infty[0,1]$ with non-negative derivative and $\chi(t)=0$ near $t=0$ and $\chi(t)=1$ near $t=1$. Again applying W\"ust's extension of the Kato-Rellich theorem, the operator $G=\tilde{D}-T+\chi\cdot T$ is a $B\otimes C[0,1]$-linear selfadjoint and regular on $\tilde{\mathpzc{E}}\otimes C[0,1]$ with domain being the graph closure of $\Dom(\tilde{D})\otimes C[0,1]$. Moreover, the local boundedness of $T$ implies that for any $a\in \mathcal{I}$ it holds that 
$$a\Dom(G)\subseteq a\Dom(\tilde{D})\otimes C[0,1]$$ 
and that $G$ has locally compact resolvent in $\tilde{\mathpzc{E}}\otimes C[0,1]$. 

We can formally construct the interior product $Q$ of $G$ with the half-closed chain $(L^2[0,1],\dmin_x)$ over $C[0,1]$ as in Lemma \ref{difftopyofoperator}. Using Theorem \ref{sumregul}, one shows that $Q$ is symmetric and regular and we obtain a bordism from $(\tilde{\mathpzc{E}}\hat{\boxtimes} L^2[0,1],Q)$ whose boundary is $-(\tilde{\mathpzc{E}},\tilde{D})+(\tilde{\mathpzc{E}},D_0)$. We have therefore constructed a bordism $(\tilde{\mathpzc{E}},\tilde{D})\sim_{\bor}(\tilde{\mathpzc{E}},D_0)$.

We define $\mathpzc{E}:=\mathpzc{E}_0$ and consider the $B$-linear unitary operator 
$$U:=\bigoplus_{j\in \N} e_{0j}:\tilde{\mathpzc{E}}\to\mathpzc{E}\otimes \ell^2(\N)= \bigoplus_{j\in \N} \mathpzc{E}.$$
It is clear that $U$ intertwines the $\mathcal{I}$-actions. We define $D:=e_{00}\tilde{D}e_{00}$ which is a regular self-adjoint operator with compact resolvent in $\mathpzc{E}$. The operator $D$ satisfies 
\begin{align*}
UD_0U^*-D\otimes \id_{\ell^2(\N)}&=\bigoplus_{j\in \N} (e_{0j}\tilde{D}e_{j0}-e_{00}\tilde{D}e_{00})\\
&=\bigoplus_{j\in \N} e_{0j}[\tilde{D},e_{j0}]e_{00}\in \End_B^*(\mathpzc{E}\otimes \ell^2(\N)),
\end{align*}
because of the uniform norm bound $\|[D,e_{j0}]\|_{\End_B^*(\tilde{\mathpzc{E}})}\leq C$ of Equation \eqref{comandlip}. It follows from Corollary \ref{bddpert} that $(\tilde{\mathpzc{E}},D_0)\sim_{\bor} (\mathpzc{E}\otimes \ell^2(\N), D\otimes \id_{\ell^2(\N)})$. In summary, there is a bordism as in \eqref{surbor}. 
\end{proof}

\begin{cor}
\label{symnormediso}
For any regular symmetrically normed ideal $\mathcal{I}\subseteq \K$, the bounded transform $b:\Omega_*^e(\mathcal{I},B)\to KK_*(\K,B)\cong K_*(B)$ is an isomorphism.
\end{cor}

\begin{appendix}
\section{A counterexample when the subalgebra is not fixed}
\label{counterdensesub}

There is a need for fixing the subalgebra in a spectral triple, or an unbounded Kasparov cycle, in order to obtain a semigroup under direct sum. For simplicity, we restrict our discussion to the unital case. To contrast the principle of ``fixing a subalgebra", we use the notion of an unbounded Fredholm module for a $C^*$-algebra $A$ as a triple $(\pi,\mathpzc{H},D)$ where $D$ is a self-adjoint operator with compact resolvent on the Hilbert space $\mathpzc{H}$ and the associated Lipschitz algebra is dense in $A$:
$$\Lip(\pi,D):=\{a\in A: a\Dom(D)\subseteq \Dom(D)\;\;\mbox{and}\;\; [D,a]\;\;\mbox{has a bounded extension}\}.$$

We will in the next proposition construct a counterexample to the statement ``the direct sum of unbounded Fredholm modules is an unbounded Fredholm module". It is likely that this result is known to experts in $KK$-theory. However, we could not find an explicit example in the literature.

We consider the following setup. Take a closed Riemannian manifold $X$ and a Dirac type operator $D$ acting on sections of a Clifford bundle $E\to X$, densely defined on $L^2(X,E)$ with domain $H^1(X,E)$. Define $\pi:C(X)\to \B(L^2(X,E))$ as point wise multiplication. If $h:X\to X$ is a homeomorphism, we set $\pi_h:=\pi\circ h^*$. The triples $(\pi_h,L^2(X,E),D)$ and $(\pi,L^2(X,E),D)$ are both unbounded Fredholm modules with $\Lip(\pi,D)=\Lip(X)$ and $\Lip(\pi_h,D)=(h^{-1})^*\Lip(X)$.

\begin{prop}
\label{hoeh}
Let $f\in C(S^1,\R)$ be a nowhere differentiable function and define
$$h:S^1\times S^1\to S^1\times S^1, \quad (z,w)\mapsto (z,\mathrm{e}^{if(z)}w).$$
The function $h$ is a homeomorphism and for $X=S^1\times S^1$ with its flat metric and $D$ the spin Dirac operator thereon, the inclusion
$$\Lip(\pi\oplus \pi_h,D\oplus D)=\Lip(\pi,D)\cap \Lip(\pi_h,D)\subseteq C(S^1\times S^1),$$
is not dense. In particular, the direct sum of unbounded Fredholm modules of the form $(\pi_h,L^2(X,E),D)$ and $(\pi,L^2(X,E),D)$ is in general not an unbounded Fredholm module for $C(X)$.
\end{prop}

\begin{proof}
The mapping $h$ is a continuous injection. It follows from the Theorem on Invariance of Domains that $h$ is a homeomorphism onto its image. As such, the image $h(S^1\times S^1)\subseteq S^1\times S^1$ is a topological submanifold of the same dimension as $S^1\times S^1$. However, since $S^1\times S^1$ is connected it holds that $h(S^1\times S^1)= S^1\times S^1$. We conclude that $h$ is a homeomorphism. 

To study $\Lip(\pi\oplus \pi_h,D\oplus D)\subseteq C(S^1\times S^1)$, we note that
$$\Lip(\pi,D)\cap \Lip(\pi_h,D)=\{a\in \Lip(S^1\times S^1): a\circ h\in \Lip(S^1\times S^1)\}.$$
We use $\partial_i$ to denote partial differentiation on $S^1\times S^1$, e.g. 
$$\partial_1a(z,w)=\frac{\mathrm{d}}{\mathrm{d} t}\bigg|_{t=0}a(z\mathrm{e}^{it},w).$$ 
For $a,a\circ h\in \Lip(S^1\times S^1)$, the following expression has a limit almost everywhere as $t\to 0$:
\begin{align*}
&\frac{a(z\mathrm{e}^{it},\mathrm{e}^{if(z\mathrm{e}^{it})}w)-a(z,\mathrm{e}^{if(z)}w)}{t}\\
&\qquad\qquad=\underbrace{\frac{a(z,\mathrm{e}^{if(z\mathrm{e}^{it})}w)-a(z,\mathrm{e}^{if(z)}w)}{t}}_{I_t(z,w)}+\underbrace{\frac{a(z\mathrm{e}^{it},\mathrm{e}^{if(z\mathrm{e}^{it})}w)-a(z,\mathrm{e}^{if(z\mathrm{e}^{it})}w)}{t}}_{II_t(z,w)}.
\end{align*}
Moreover, since $a$ is Lipschitz, we can by Egorov's theorem deduce that for any $\epsilon>0$, $II_t(z,w)$ converges to $\partial_1a(z,\mathrm{e}^{if(z)}w)$ as $t\to 0$ on the complement of a set of measure $\epsilon$. Therefore, for any $\epsilon>0$, $I_t(z,w)$ converges as $t\to 0$ on the complement of a set of measure $\epsilon$. However, we have 
$$I_t(z,w)=\underbrace{\frac{a(z,\mathrm{e}^{if(z\mathrm{e}^{it})}w)-a(z,\mathrm{e}^{if(z)}w)}{f(z\mathrm{e}^{it})-f(z)}}_{A_t(z,w)}\cdot \underbrace{\frac{{f(z\mathrm{e}^{it})-f(z)}}{t}}_{B_t(z,w)}.$$
Using that $a$ is Lipschitz, $A_t(z,w)$ converges almost everywhere to $\partial_2a(z,\mathrm{e}^{if(z)}w)$ as $t\to 0$. Since $f$ is nowhere differentiable, $B_t(z,w)$ converges nowhere and we deduce that for all $\epsilon$, $\partial_2a(z,\mathrm{e}^{if(z)}w)=0$ on the complement of a set of measure $\epsilon$. We conclude that any $a\in \Lip(\pi,D)\cap \Lip(\pi_h,D)$ is constant in the second variable, i.e. $a(z,w)=a(z,1)$ for any $(z,w)$. Therefore $\Lip(\pi\oplus \pi_h,D\oplus D)\subseteq C(S^1\times S^1)$ is not dense.
\end{proof}

\end{appendix}

\paragraph{\textbf{Acknowledgements}}
Foremost we express our gratitude to Michel Hilsum and Matthias Lesch for valuable and inspiring discussions. The first listed author was supported by the ANR Project SingStar. The second listed author thanks the Knut and Alice Wallenberg foundation for their support. The second listed author was supported by the Swedish Research Council Grant 2015-00137 and Marie Sklodowska Curie Actions, Cofund, Project INCA 600398. The third listed author was supported by the EPSRC grant EP/J006580/2. We also thank the Universit\'e Blaise Pascal Clermont-Ferrand, Leibniz Universit\"at Hannover and the Graduiertenkolleg 1463 (\emph{Analysis, Geometry and String Theory}) and the University of Warwick for facilitating this collaboration. We are grateful to the Hausdorff Institute for Mathematics in Bonn (trimester program on \emph{Noncommutative Geometry and its Applications}) and the University of Wollongong for their hospitality. The authors first met at ``The sixth annual Spring Institute on Noncommutative Geometry and Operator Algebras" at Vanderbilt University; this collaboration would not have been possible without this event. We thank Simon Brain, Jens Kaad and Adam Rennie for fruitful discussions. We also thank the referee for a number of useful comments.

\end{document}